\documentclass{article}
\usepackage[utf8]{inputenc}
\usepackage{amsmath, amssymb,amsthm,amsfonts,mathrsfs}
\usepackage{fullpage}
\usepackage{xcolor}
\usepackage{dsfont}
\usepackage{graphicx}
\usepackage[font=small,skip=0pt]{caption}
\theoremstyle{plain}

\newcommand{\E}{\mathbb{E}}

\newcommand{\N}{\mathbb{N}}
\renewcommand{\P}{\mathbb{P}}
\newcommand{\Q}{\mathbb{Q}}
\newcommand{\R}{\mathbb{R}}

\newcommand{\EE}{\mathcal{E}}
\newcommand{\FF}{\mathcal{F}}

\newcommand{\HH}{\mathcal{H}}

\newcommand{\KK}{\mathcal{K}}

\newcommand{\PPP}{\mathscr{P}}

\newcommand{\given}{\,|\,}
\newcommand{\sbgiven}{\,\big|\,}
\newcommand{\bgiven}{\,\Big|\,}
\newcommand{\supp}{\textnormal{supp}}
\newcommand{\law}{\textsf{Law}}
\newcommand{\wto}{\stackrel{w}{\Longrightarrow}}
\newcommand{\eps}{\varepsilon}
\newcommand{\one}{\mathds{1}}
\newcommand{\la}{\lambda}
\newcommand{\normal}{\mathsf{N}}
\newcommand{\rd}{\mathrm{d}}
\def\sT{{\sf T}}
\newcommand{\Ber}{\textnormal{Ber}}
\newcommand{\Unif}{\textnormal{Unif}}
\newcommand{\Beta}{\textnormal{Beta}}
\newcommand{\wt}[1]{\widetilde{#1}}

\DeclareMathOperator{\Var}{Var}
\DeclareMathOperator{\Cov}{Cov}

\newtheorem{thm}{Theorem}[section]
\newtheorem{prop}[thm]{Proposition}
\newtheorem{cor}[thm]{Corollary}
\newtheorem{lemma}[thm]{Lemma}

\theoremstyle{definition}
\newtheorem{defn}[thm]{Definition}
\newtheorem{eg}[thm]{Example}

\title{Agreement and Statistical Efficiency in Bayesian Perception Models}
\author{Yash Deshpande\thanks{This work was conducted while YD was at 
the Department of Mathematics, Massachusetts Institute of Technology. 
Supported  by ARO MURI W911NF1910217}
\and Elchanan Mossel\thanks{Department of Mathematics and IDSS, Massachusetts Institute of Technology. Supported by Simons-NSF collaboration on deep learning NSF DMS-2031883, by ARO MURI W911NF1910217, by 
Vannevar Bush Faculty Fellowship award ONR-N00014-20-1-2826 and by a Simons Investigator Award in Mathematics (622132)} \and 
Youngtak Sohn \thanks{Department of Mathematics Massachusetts Institute of Technology. Supported by Simons-NSF collaboration on deep learning NSF DMS-2031883, by
Vannevar Bush Faculty Fellowship award ONR-N00014-20-1-2826}}
\begin{document}

\maketitle
\begin{abstract}
 Bayesian models of group learning are studied in Economics since the 1970s 
and more recently in computational linguistics. The models from Economics postulate that agents maximize utility in their communication and actions. 
The Economics models do not explain the ``probability matching" phenomena that are observed in many experimental studies. 
To address these observations, Bayesian models that do not  formally fit into the economic utility maximization framework were introduced. 
In these models individuals sample from their posteriors in communication. 
In this work we study the asymptotic behavior of such models on connected networks with repeated communication. 
Perhaps surprisingly, despite the fact that individual agents are not utility maximizers in the classical sense, we establish that the individuals ultimately agree and furthermore show that the limiting posterior is Bayes optimal.

We explore the interpretation of our results in terms of Large Language Models (\textsc{LLM}s). In the positive direction our results can be interpreted as stating that interaction between different LLMs can lead to optimal learning.
However, we provide an example showing how misspecification may lead LLM agents to be overconfident in their estimates.  
\end{abstract}
\section{Introduction}\label{sec:intro}

How do agents learn from each other? 

In this work we are interested in Bayesian models of learning with the goal of bringing together perspectives of theoretical Economics, experimental and computational psychology, and language models.  
We are interested in scenarios where agents, either humans or artificial intelligence, repeatedly observe each other and communicate on a connected graph representing the social network connecting them.

Since the foundational work of Aumann~\cite{Aumann:76}, economists have modeled collections of rational individuals as Bayesian agents who learn from each other and take actions as to maximize utility. The network setup was introduced in~\cite{GeanakoplosPolemarchakis:82} and the results of~\cite{Aumann:76,GeanakoplosPolemarchakis:82} establish that agents ultimately agree in their posterior if each action reveals the current individual posterior mean. We will call this the {\em belief model}. More recent research~\cite{DemarzoSkiadas:99,Ostrovsky:12}, see also ~\cite{MosselTamuz:17,MMSO:18}, establish that assuming that the initial information the agents receive is independent (conditionally given the unknown state of the world), then the posterior is the optimal Bayes posterior given all private information. 

Another relevant body of work in Economics involve scenarios where the agents are again utility maximizers but their actions do not reveal their posterior exactly. This happens for example if the space of possible actions that agents can take is finite. In this case, which we will call the {\em action model}, it is known that agents asymptotically agree~\cite{GaleKariv:03,RoSoVi:09,MoSlTa:14b}. However, in general, the agents do not reach the optimal posterior given all private signals. Learning results are only established if the number of individuals goes to infinity (with additional conditions)~\cite{MoSlTa:14b,MoSlTa:15c}. 

Work in psychology, see e.g. the surveys~\cite{Myers:14,Vulkan:00}, found that individuals do not follow the utility maximizing paradigm. 
Thus their behavior fits neither the action models nor the belief models. 
Instead, they apply ``probability matching" - taking actions based on samples from their posterior rather than based on its mean. 
This was formalized as a Bayesian model in~\cite{GriffithsTenenbaum:06,GoTeFeGr:08,VuGoGrTe:14} among other works.  
We call these models {\em sampling models}. There is a large body of work that aims to justify various sampling models and their variants as models how individuals communicate, see e.g.~\cite{VATB:09,SaGrNa:10,VuGoGrTe:14,KMRB:10}
 and the survey~\cite{GeHoTe:15}. However, to the best of our knowledge, the sampling models were never theoretically analyzed when there is more than 1 agent. In this work we study the sampling model and mathematically establish the consequences in terms of agreement and learning.

Furthermore, we reinterpret the sampling models as an abstraction of large language models (LLMs) communicating with each other. It is well-known that LLMs iteratively sample posteriors given the data rather than maximizing utility. Moreover, now that many LLMs are in use, a natural concern is that the output of one LLM may be used by others and this will lead to 
lower quality performance (a related concern has been discussed in \cite{Pan2023risk}). In this work, we show that indeed, if we model two LLMs as communicating Bayesian agents, then the missepcification of the posterior structure leads to suboptimal performance. Further explanation is given in the next section.  


\subsection{A theoretical model of communicating LLMs}



To connect our study with the models in this paper, we abstractly consider each LLM as a Bayesian agent. We focus on the conceptual aspects of their behavior and assume that they sample from learned distributions. Let's consider two such agents: LLM1 and LLM2. Each has observed its initial training data (initial signal). Now, envision a conversation between LLM1 and LLM2 on a specific question, such as ``when will LLMs write better ML papers than tenured faculty?"

As Bayesian agents, both LLM1 and LLM2 answer the question by sampling from their respective posteriors. However, they continuously update their models based on the responses received from each other. We explore whether this iterative process leads to improved answers and better learning for both LLMs.

Our findings suggest that if each LLM has a good understanding of how the other LLM works and they are trained on independent data, their conversation will indeed lead to improvements in their posterior distributions, resulting in more accurate answers.

On the other hand, we also examine a scenario where the LLMs are misspecified, treating each other's responses as new independent samples generated to answer the question. For example, they may view each others answer as coming from a new relevant web page or book that addresses the question. In this case, we show that both LLMs will converge to the same suboptimal posterior, and that they will be overconfident in their estimates. In fact, our analysis in Section \ref{subsec:misspecified} demonstrates that they will converge to a point mass, where the location of the point mass is random and has zero probability of being equal to the optimal posterior mean.

\section{Models and main results}
\label{sec:model}
In this section we present the main models and results of the paper. 

In Section~\ref{subsec:model}, we define the sampling model and prove the agreement and statistically efficient learning for the model. Namely, we consider a network of agents, where the agents compute their posterior distribution exactly but only communicate samples to their neighbors. We prove that {\em for any given finite graph, all agents agree asymptotically} and moreover that they {\em all converge to the optimal Bayes posterior given all the data} under mild assumptions on the signal structure. For the exact statement, see Theorems \ref{thm:agreement} and \ref{thm:eff:learning} below. Thus for sampling model, both agreement and statistical efficiency (``learning") hold. We thus prove analogs of the previously mentioned theoretical Economics results for the belief models in the sampling model.

In Section~\ref{subsec:Bernoulli}, we study a variant of the sampling model, which we call \textit{Bernoulli sampling model}, where the agents only communicate $0$ and $1$ based on their samples from posteriors. We show that the analogous results for the sampling model hold here. That is, agents agree on the posterior mean, which is the sufficient statistics for the Bernoulli sampling model, and the agreed posterior mean is given by the optimal posterior mean.

In Section~\ref{subsec:misspecified}, we analyze the Bernoulli sampling model in a misspecified situation, where there are $2$ agents and each agent believes that the samples they see from the other agent are independent coin tosses with the original unknown bias. We show that in this misspecified setting, the agents will eventually agree, however, their posteriors will converge to a point mass. Further, we show that the location of the point mass does not equal the mean of the optimal posterior mean, almost surely. 

As we mentioned earlier, a major motivation for studying the (Bernoulli) sampling model is to abstractly study interactions between different LLMs. In this abstraction, the positive results show that if the LLMS are trained on disjoint sets of data and are aware that they are communicating with other LLMs, they learn better. 
However, the misspecified case, they learn worse and are overconfident in their results.

\subsection{The sampling model}\label{subsec:model}
We consider the following \textit{sampling model} for social learning on a network. There is
an unknown state $\theta\in \R^d$ drawn from a common prior $\nu$
on $\R^d$. Agents live at the vertices of a network $G = (V, E)$, where we
identify $V \equiv [n]$. At time $t=0$, agents obtain the initial private signals
$S_i \sim \P_i (\cdot | \theta)\in \PPP(\R^d)$, conditionally independent given $\theta$. We call $\{\P_i(\cdot |\theta)\}_{i\in V}$ the \textit{signal structure} of the sampling model.

At time steps $t\in \N$, each agent $i\in V$ exchanges information with their neighbors $\partial i\subseteq V$ as follows. At time $t=1$, denote by $\nu_i(1):=\law (\theta \given S_i)$ the posterior distribution of $\theta$ given the initial information $\HH_i(1):= \sigma(S_i)$ of agent $i$, where $\sigma(S_i)$ is the $\sigma$-algebra generated by $S_i$. Agent $i$ then samples $Y_i(1)\sim \nu_i(1)$ given $\HH_i(1)$, conditionally independent of the other agents, and communicates $Y_i(1)$ to all of its neighbor $j\in \partial i$. More generally, at time $t\in \N$, denote the information agent $i$ has seen at time $t$ by the $\sigma$-algebra
\begin{equation*}
    \HH_i(t):= \sigma(S_i,(Y_j(s))_{j\in \partial i, s<t})
\end{equation*}
Then, at time $t$, agent $i$ samples $Y_{i}(t) \sim \nu_{i}(t):=\law(\theta\given \HH_i(t))$ given $\HH_i(t)$, conditionally independent of all other information, and communicates $Y_i(t)$ to its neighbors.

We consider the following social learning questions:
\begin{enumerate}
\item Do the posterior beliefs $\nu_{i}(t)$ converge to a common posterior
$\nu(\infty)$, as $t\to\infty$, independently of $i$?
\item If so, does the common posterior belief $\nu(\infty)$ coincide with the optimal Bayes posterior given all private information $\law(\theta \given S_1,...,S_n)$?
\end{enumerate}
Our first theorem establishes the positive answer to the first question under the natural assumption that the network $G$ is connected. Denote by $\HH_i(\infty):=\cup_{t\geq 1}\HH_i(t)$ for $i\in V$.
\begin{thm}\label{thm:agreement}
For every agent $i\in V$, the sequence of its (random) posterior distribution $\nu_i(t)$ converges weakly to $\nu_i(\infty):=\law(\theta \given \HH_i(\infty))$, almost surely. Furthermore, if agents $i$ and $j$ are connected, then $\nu_i(\infty)=\nu_j(\infty)$ holds, almost surely. Therefore, provided that $G$ is connected, all agents in $G$ agree asymptotically.
\end{thm}

The proof of Theorem \ref{thm:agreement} is given in Section \ref{sec:proof}. The proof is based on the martingale convergence theorem and on a sampling analog of the \textit{imitation principle}. 
It states that if in the limit the neighboring agents $i$ and $j$ each converge to a different posteriors, then by observing each other they can learn the other agents posterior and then improve their own posterior (the argument uses convexity). 

Observe that by the nature of the model, the agents $i,j\in V$ do not communicate with each other if $i$ and $j$ are not connected in $G$. Thus, when $G$ is not connected, we can apply the theorem above to every connected component in $G$, and conclude that agents asymptotically agree in each connected component.

Next, we show that the answer to the second question is also positive under mild assumptions on the signal structure $\{\P_i(\cdot \given \theta)\}_{i\in V}$, which we now state.

\begin{defn}[Signal structure satisfying the Bayes law]\label{def:bayes}
We say $\{\P_i(\cdot |\theta)\}_{i\in V}$ satisfies the Bayes law if there exists a deterministic $f:(\R^d)^{n}\to \R_{>0}$ such that for every Borel measurable $A\subseteq \R^d$ satisfying $\P(\theta\in A)\in (0,1)$ and conditionally independent signals $S_i\given \theta \sim \P_i(\cdot\given \theta)$, we have
\begin{equation}\label{eq:bayeslaw}
    \P(\theta \in A\,|\, S_1,\ldots,S_n)=\frac{\prod_{i=1}^{n}\P(\theta\in A\,|\,S_i)}{\P(\theta\in A)^{n-1}}\cdot f(S_1,..,S_n)\quad\textnormal{almost surely.}
\end{equation}
\end{defn}
A crucial remark is that the Definition \ref{def:bayes} is a natural generalization of the usual Bayes law to the case where $S_i$ may be continuous. When the signals $\{S_i\}_{i\in V}$ are discrete, the equation \eqref{eq:bayeslaw} is easily satisfied: by the conditional independence of $S_i$ given $\theta$, if $\P(\theta \in A)\neq 0$,
\begin{equation*}
\begin{split}
\P(\theta \in A \given S_1,\ldots, S_n)
&=\frac{\Big(\prod_{i=1}^{n}\P(S_i=s_i\given \theta \in A)\Big)\cdot\P(\theta\in A)}{\P(S_1=s_1,\ldots, S_n=s_n)}\\
&=\frac{\prod_{i=1}^{n}\P(\theta\in A\,|\,S_i)}{\P(\theta\in A)^{n-1}}\cdot \frac{\prod_{i=1}^{n}\P(S_i=s_i)}{\P(S_1=s_1,\ldots, S_n=s_n)}\,.
\end{split}
\end{equation*}
When $S_i$'s are not discrete, see Lemma \ref{prop:suff} below for a general condition for which the Bayes law holds.
\begin{defn}[Signal structure with finite information]\label{def:finite}
We say $\P_i(\cdot\given\theta)$ has finite information if for every Borel measurable $A\subseteq \R^d$ satisfying $\P(\theta\in A)\in (0,1)$ and conditionally independent signals $S_i\given \theta \sim \P_i(\cdot\given \theta)$, we have
\begin{equation}\label{eq:finite:info}
    \E\left|\log\left(\frac{\P(\theta\in A\given S_i)}{\P(\theta\in A^{c}\given S_i)}\right)\right|<\infty,
\end{equation}
where the expectation is over marginal distribution over $S_i$. We say $\{\P_{i}(\cdot \given \theta))\}_{i\in V}$ has finite information if for every agent $i\in V$, $\P_i(\cdot\given \theta)$ has finite information.
\end{defn}
The equation \eqref{eq:finite:info} can be interpreted as follows. Note that without conditioning on $S_i$, then the quantity in \eqref{eq:finite:info} is always finite for $A$ such that $\P(\theta \in A)\in (0,1)$. Thus, if further conditioning on $S_i$ typically does not introduce dramatic changes, then \eqref{eq:finite:info} is satisfied. We demonstrate in Section \ref{subsec:examples} that many natural examples satisfy Definitions \ref{def:bayes} and \ref{def:finite}. Under these conditions, we establish statistically efficient learning.

\begin{thm}[Statistically efficient learning]\label{thm:eff:learning}
Suppose that $\left\{\P_i(\cdot\given \theta)\right\}_{i\in V}$ satisfies the Bayes law (cf. Definition \ref{def:bayes}) and has finite information (cf. Definition \ref{def:finite}). Assuming that the graph $G$ is connected, in the limit, every agent learns as well as if they learned all
agents' private information. That is, the common limiting posterior $\nu(\infty)\equiv \nu_i(\infty)$ coincides with $\law(\theta\given S_1, \ldots, S_n)$, almost surely.
\end{thm}
We prove Theorem \ref{thm:eff:learning} in Section \ref{sec:proof} based on a delicate argument using the Bayes law in Definition \ref{def:bayes} and Chebyshev's sum inequality. The assumption regarding finite information (cf. Definition \ref{def:finite}) is needed to guarantee that the random variables that appear in the argument has well-defined expectation. 
\subsubsection{Examples of signal structures with finite information}
\label{subsec:examples}

\begin{eg}
    Suppose that the signals are supported on a finite set, i.e. $\P_i(\cdot \given \theta)$ is supported on $\mathcal{S}_i$ for any $i\in V$. Then, the signal structure $\{\P_i(\cdot \given \theta)\}_{i\in V}$ satisfied the Bayes law and has finite information. Indeed, as mentioned above, since $\mathcal{S}_i$ is discrete, the signal structure satisfies Bayes law. Further, note that since the support is finite for any $i\in V$, there exists constants $c_i,C_i>0$ such that $\P(S_i=s_i\given \theta) \in [c_i,C_i]$ holds for any $s_i$ in the support. Thus, for every $A\subseteq \R^{d}$ such that $\P(\theta \in A)\in (0,1)$, it follows that
    \begin{equation*}
\frac{\P(\theta \in A \given S_i=s_i)}{\P(\theta \in A)}=\frac{\int_{\theta \in A}\P(S_i=s_i\given \theta)\nu(\rd \theta)}{\P(S_i=s_i)\P(\theta \in A)}\in [c_i,C_i]\,.
\end{equation*}
Thus, we have that 
\begin{equation*}
 \E\left|\log\left(\frac{\P(\theta\in A\given S_i)}{\P(\theta\in A^{c}\given S_i)}\right)\right|\leq  \E\left|\log\left(\frac{\P(\theta\in A)}{\P(\theta\in A^{c})}\right)\right|+\log\left(\frac{C_i}{c_i}\right)<\infty\,.
\end{equation*}
Therefore, if $\P_i(\cdot\given \theta)$ is supported on a finite set for every $i\in V$, it also has finite information (cf. Definition~\ref{def:finite}), thus the assumptions of Theorem \ref{thm:eff:learning} is satisfied.
\end{eg}
We now consider the examples where the signals $S_i$ may be continuous. In Appendix \ref{sec:omittedproofs}, we prove the following sufficient condition. We first recall the definition of $\sigma$-finite regular Borel measures (see e.g. Definition 2.5.12 of \cite{BeCz09}).
\begin{lemma}\label{prop:suff}
Suppose that there exist $\sigma$-finite regular Borel measures on $\R^d$, $\left\{\mu_i\right\}_{i \in V}$, such that the $\P_i(\cdot\given \theta)$ is absolutely continuous with respect to $\mu_i$ for every $\theta \in \supp(\nu)$ and $i\in V$. Then, $\left\{\P_i(\cdot\given \theta)\right\}_{i\in V}$ satisfies the Bayes law in Definition \ref{def:bayes}. Moreover, in this case, the posterior $\law(\theta\given S_i)$ is absolutely continuous with respect to the prior $\nu$, $\P_{S_i}$-a.s. for $i\in V$, where $\P_{S_i}$ denote the marginal law of $S_i$. Denote $p_i(x\given S_i)$ by its the Radon-Nikodym derivative. Suppose that the following condition is further satisfied: for every $i\in V$, 
\begin{equation}\label{eq:condition:integrable}
    \E_{S_i}\int_{\R^{d}}\Big|\log \left(p_i\left(x\given S_i\right)\right)\Big| \nu (\rd x) <\infty.
\end{equation}
Then $\{\P_i(\cdot\given \theta)\}_{i\in V} $ has finite information per Definition \ref{def:finite}.
\end{lemma}
In several examples the integrability condition \eqref{eq:condition:integrable} is easy to verify as seen below.

\begin{eg}\label{eg:gaussian}
    In Appendix \ref{sec:gaussian}, we consider the Gaussian model, where the prior is Gaussian $\nu=\normal(\xi,\Sigma)$ for some $\xi\in \R^{d}$ and p.s.d. matrix $\Sigma\in \R^{d\times d}$, and the signal structure given by $\P_i(\cdot\given \theta)= \normal(\langle a_i,\theta\rangle, \Sigma_i)$, where $a_i\in \R^{d}$. This Gaussian model satisfies the condition for Lemma \ref{prop:suff} with $\mu_i\equiv \mu, i \in V$, where $\mu$ is the Lebesgue measure on $\R^{d}$ since we can simply bound
\begin{equation*}
    \Big|\log\left(p_i(x \given S_i)\right)\Big|=\left|\log\frac{\rd \law(\theta\given S_i)}{\rd \mu}(x)-\log\frac{\rd\nu}{\rd \mu}(x)\right|\leq C\big(\|x\|_2^2+\|S_i\|_2^2\big),
\end{equation*}
where $C>0$ is a constant depending only on $a_i$. Thus, the condition \eqref{eq:condition:integrable} holds, so by Lemma \ref{prop:suff}, the Gaussian signal structure has finite information and satisfies the Bayes law. 
\end{eg}

\begin{eg}\label{eg:lr:bounded}
Another example satisfying the integrability condition \eqref{eq:condition:integrable} is given as follows. Suppose that $\theta\in \{0,1\}$ is binary with $p_x:= \P(\theta=x)>0$ for $z\in (0,1)$, and assume that $\mathbb{P}_i(\cdot\mid\theta)$ has density $q_i(x\mid\theta)$ with respect to a $\sigma$-finite Borel measure $\mu(dx)$. By a classical Bayes law (see e.g. Theorem 1.31 in \cite{Schervish95}),
\begin{equation*}
    p_i(x\given S_i)= \frac{q_i(S_i\given x)}{q_i(S_i\given 0)p_0+q_i(S_i \given 1)p_1}\,.
\end{equation*}
Thus, if we denote the likelihood ratio by $L_i(s)\equiv \frac{q_i(s\mid 1)}{q_i(s\mid 0)}$, then \eqref{eq:condition:integrable} is equivalent to
\begin{equation*}
    \mathbb{E} \Big[L_i(S_i)\vee L_i(S_i)^{-1}\Big]<\infty\,.
\end{equation*} 
Thus, if the likelihood ratio $L_i(x)$ is bounded away from $0$ and $\infty$, then the signal structure has finite information by Lemma \ref{prop:suff}. Thus, in this case, the conditions of Theorem \ref{thm:eff:learning} is satisfied.
\end{eg}

\subsection{Bernoulli sampling model}
\label{subsec:Bernoulli}
In this section, we consider a variant of the sampling model, which we call \textit{Bernoulli sampling model} as follows. Let $d=1$ and let the unknown state $\theta\in [0,1]$ drawn from the uniform prior $\theta \in \textnormal{Unif}([0,1])$.

At time $t=0$, agent $i\in V$ obtain initial private signal $S_i\sim \Ber(\theta)$ independently, i.e. flipping a coin with probability $\theta$. Then, at time steps $t\geq 1$, each agent $i\in V$ exchanges information with their neighbors $\partial i\subseteq V$ as follows. Denote $\nu_i(1):=\law (\theta \given S_i)$ and $\KK_i(1):= \sigma(S_i)$. At time $t=1$, agent $i$ samples $Y_i(1)\sim \nu_i(1)$ given $\KK_i(1)$, conditionally independent of the other agents, and further samples $X_i(1)\sim \Ber(Y_i(1))$, conditionally independent given other random variables. Then, agent $i$ communicates $X_i(1)$ to all of its neighbor $j\in \partial i$. More generally, at time $t\in \N$, denote the information agent $i$ has seen at time $t$ by the $\sigma$-algebra
\begin{equation*}
    \KK_i(t):= \sigma(S_i,(X_j(s))_{j\in \partial i, s<t})
\end{equation*}
Then, at time $t$, agent $i$ samples $Y_{i}(t) \sim \nu_{i}(t):=\law(\theta\given \KK_i(t))$ and $X_i(t) \sim \Ber\big(Y_i(t)\big)$ given $\KK_i(t)$, conditionally independent of all other information, and communicates $X_i(t)$ to its neighbors. 

Thus, compared to the sampling model, agents only communicate $\{0,1\}$ valued information $X_i(t)$ instead of actual sample $Y_i(t)$ from the posterior by flipping a coin with probability $Y_i(t)$. In this Bernoulli sampling model, since agents only communicate $X_i(t)\sim \Ber (Y_i(t))$ instead of $Y_i(t)$, the best hope is for the agents to learn the mean of the optimal Bayes posterior $\E[\theta\mid S_1,\ldots, S_n]$. More precisely, denote the posterior mean of agent $i$ at time $t\geq 0$ by
\begin{equation*}
m_i(t):= \E\big[\theta\,\big|\,\KK_i(t)\big]\,.
\end{equation*}
Then, we establish the agreement and efficient statistical learning of the posterior means. As before, we assume without loss of generality that the network $G$ is connected since otherwise, we can apply our result for every connected component of $G$. 
\begin{thm}\label{thm:bernoulli:sampling}
Under the Bernoulli sampling model on the connected graph $G$ as above. Then, for every agent $i\in V$, the (random) posterior mean $m_i(t)$ converges as $t\to\infty$ to $m_i(\infty)$, almost surely and in $L^1$. Furthermore, the limiting posterior mean $m_i(\infty)$ is identical for every agent $i\in V$, which is given by the mean of the optimal Bayes posterior: we have almost surely that for every $i\in V$,
\[
m_i(\infty) = \E[\theta\mid S_1,\ldots, S_n]\,.
\]
\end{thm}
\subsection{Misspecified Bernoulli sampling model}
\label{subsec:misspecified}
Theorems \ref{thm:eff:learning} and \ref{thm:bernoulli:sampling} crucially rely on the fact that each agent knows which agents they communicate with and further more, know how those agents compute their distributions based on their initial signal and previous rounds of communication. 
It is natural to ask if the results still hold if agents are more naive, and update their posterior assuming each round of communication is drawn from the original signals.
As mentioned in the introduction this is partially motivated by AI agents, such as LLMs who may not be aware that some of their training data is generated by other LLMs. 

One might ask the following question:

\textit{Do similar learning results hold if agents assume that the communication they are seeing are just indpendent samples from the signal distribution?}

As it turns out, to achieve statistically efficient learning as in Theorem \ref{thm:eff:learning} and Theorem \ref{thm:bernoulli:sampling}, it is indeed important that agents know how the communication that they observe is generated. 

To demonstrate this, we consider the \textit{misspecified Bernoulli sampling model} as follows. For simplicity, assume that the underlying graph $G$ is an edge, i.e. $V=\{1,2\}$ and the agents $1$ and $2$ are connected. As before, at time $t=0$, agent $i\in V$ obtain their signals $S_i \sim \Ber(\theta)$ independently given $\theta$. At time steps $t\geq 1$, however, agent $1$ and $2$ communicates as follows. Agent $1$ thinks that agent $2$ is sending information that are i.i.d. $\Ber(\theta)$. That is, for i.i.d. $\Ber(\theta)$ random variables $Z_1,\ldots, Z_m$ given $\theta$, if we denote 
\begin{equation}\label{eq:posterior:beta}
\wt{p}(\theta\given z_1,\ldots z_m):= \law(\theta \given Z_1=z_1,\ldots Z_m=z_m)\,,\quad z_1,\ldots z_m\in \{0,1\}\,,
\end{equation}
 then at time $t\geq 1$, agent $1$ samples $\wt{Y}_{1}(t)\sim \wt{p}\big(\theta\given S_{1},(\wt{X}_{2}(s))_{s<t}\big)$ conditionally independent given other random variables, and communicates $\wt{X}_1(t)\sim \Ber(\wt{Y}_1(t))$ to agent $2$. Analogously, agent $2$ at time $t$ samples $\wt{Y}_2(t)\sim \wt{p}\big(\theta\given S_2, (\wt{X}_1(s))_{s<t}\big)$ conditionally independent given other random variables, and communicates $\wt{X}_2(t)\sim \Ber(\wt{Y}_2(t))$ to agent $1$. Under this misspecified Bernoulli sampling model, consider the (misspecified) posteriors 
\[
\wt{\nu}_1(t):=\wt{p}\big(\theta\given S_{1},(\wt{X}_{2}(s))_{s<t}\big)\,,\quad~~\wt{\nu}_2(t):=\wt{p}\big(\theta\given S_2, (\wt{X}_1(s))_{s<t}\big)\,.
\]
We show that $\wt{\nu}_1(t)$ and $\wt{\nu}_2(t)$ agree asymptotically, however, on a point mass, where the point mass does not match the mean of the optimal Bayes posterior \textit{almost surely}. Thus, in this model, the agents are \textit{overconfident} about their beliefs, which differs from the optimal Bayes posterior.

\begin{prop}\label{prop:misspecified:bernoulli}
Consider the misspecified Bernoulli sampling model as above. Then, as $t\to\infty$, we have almost surely for $i=1,2,$ that
\[
\wt{\nu}_i(t)\stackrel{w}{\longrightarrow}\delta_{M(\infty)}\,,
\]
where $\stackrel{w}{\longrightarrow}$ denotes the weak convergence in $\PPP([0,1])$, and $M(\infty)\in [0,1]$ is a random variable that does not depend on the agents $i\in V$. Furthermore, $M(\infty)$ differs from $\E[\theta\given S_1,\ldots S_n]$ almost surely, i.e. 
\[
\P\big(M(\infty)\neq\E[\theta\given S_1,\ldots S_n])=1\,.
\]
\end{prop}
We remark that, interestingly, the misspecified Bernoulli sampling model can be considered as a variant of the classical P\'{o}lya urn model \cite{Polya31}. To see this, note that the misspecified posterior $\tilde{p}(\theta \given z_1,\ldots, z_m)$ in \eqref{eq:posterior:beta} is the beta distribution with parameters $\alpha= \sum_{i=1}^{m}Z_i+1$ and $\beta=m+1-\sum_{i=1}^{m}Z_i$ whose mean is given by $\frac{\sum_{i=1}^{m}Z_i+1}{m+2}$ (for details, see Lemma \ref{lem:beta}). Thus, conditional on the information $(S_1,\wt{X}_2(s))_{s<t}$ that agent $1$ has received up to time $t$, the message $X_1(t)$ that agent $1$ sends to agent $2$ at time $t$ is distributed as 
\[
X_1(t)\sim \Ber\Big(\frac{S_1+\sum_{s=1}^{t-1}\wt{X}_2(s)+1}{t+2}\Big)\,.
\]
Similarly, at time $t$, conditional on the past, $X_2(t)$ is distributed as $\Ber\big(\frac{S_2+\sum_{s=1}^{t-1}\wt{X}_1(s)+1}{t+2}\big)$. 

As a consequence, the evolution of the misspecified Bernoulli sampling model is the same as the following \textit{interacting urn system}, which might be of independent interest. Given $\theta\in \Unif([0,1])$, first draw $S_1,S_2\sim \Ber(\theta)$ independently. At time $t=0$, there are $2$ urns, each urn containing $S_i+1$ red balls and $2-S_i$ blue balls for $i=1,2$. At time $t=1$, we draw a ball from each urn simultaneously uniformly at random within the urn, i.e. with probability $(S_i+1)/3$, we draw a red ball from $i$'th urn, and otherwise we draw blue ball. Then, send the ball drawn from $i$'th urn to the other urn for $i=1,2$. We repeat this process for $t\geq 2$.

Although many generalizations of the P\'{o}lya urn model \cite{Polya31} have been studied in the literature (e.g. see \cite{Pemantle07} for a survey), we could not find previous works on this interacting urn system. Thus, Proposition \ref{prop:misspecified:bernoulli} interpreted as interacting urns may be interesting in its own right. The proof of Proposition \ref{prop:misspecified:bernoulli} applies idea  from the analysis of Bernard Friedman's urn model studied in \cite{Freedman65} as well as the stochastic approximation method~\cite{Pemantle90}

\section{Proofs}\label{sec:proof}
In this section, we prove our main results. In Section \ref{subsec:proof:agreement}, we prove Theorem \ref{thm:agreement}, while deferring some technical steps to Appendix \ref{sec:omittedproofs}. In Section \ref{subsec:proof:eff:learning}, we prove Theorem \ref{thm:eff:learning}. In Section \ref{subsec:proof:misspecified}, we prove Theorem \ref{thm:bernoulli:sampling} and Proposition \ref{prop:misspecified:bernoulli}. The proof of Lemma \ref{prop:suff} is also deferred to Appendix \ref{sec:omittedproofs}.

\subsection{Proof of Theorem \ref{thm:agreement}}
\label{subsec:proof:agreement}
The proof of the first claim of Theorem \ref{thm:agreement} regarding the convergence of the posterior $\nu_i(t)$ is an application of the martingale convergence theorem: for $i\in V$ and $t\in \N \cup\{\infty\}$, denote the distribution function of $\nu_i(t)$ by $F_{it}(u):= \P\left(\theta\leq u \given \HH_i(t)\right)$ for $u\in \R^{d}$. Here, we denoted $a\leq b$ for $a, b\in \R^{d}$ if coordinate-wise $a_i\leq b_i, 1\leq i\leq d$ holds. Note that by L\'{e}vy's upward theorem (see for e.g. Theorem 4.6.8 of \cite{Durrett19}), for fixed $u\in \R^{d}$, we have
\begin{equation*}
    F_{it}(u) \to F_{i\infty}(u):=\P\left(\theta \leq u \given \HH_i(\infty)\right)\textnormal{  as  } t\to \infty,\textnormal{ a.s. and in $L^1$}.
\end{equation*}
Since $\Q^{d}$ is countable, the convergence above holds simultaneously for $u\in \Q^{d}$: almost surely, we have
\begin{equation}\label{eq:convergence:distribution}
    F_{it}(u) \to F_{i\infty}(u)\textnormal{  as  } t\to \infty\textnormal{ for every $u\in \Q^{d}$}.
\end{equation}
Then, we can argue that on the a.s. event that \eqref{eq:convergence:distribution} happens, $\nu_i(t)\wto\nu_i(\infty)$ holds. We include the full detail in Appendix \ref{sec:omittedproofs} (see Proposition \ref{lem:genconvergence}). The proof of the second claim of Theorem \ref{thm:agreement} regarding agreement is more challenging. We use the following lemma whose proof is given in Appendix \ref{sec:omittedproofs}.
\begin{lemma} \label{lem:avghelps}
 Consider two random variables $Z_1, Z_2$, 
with finite second moments. If their average $Z(\lambda) = \lambda Z_1 + (1-\lambda)Z_2$ satisfies $\min_{\lambda\in[0,1]}  \E\left[Z(\lambda)^2\right] =  \E\left[Z_1^2\right]=
\E\left[Z_2^2\right]$ then $Z_1 = Z_2$ almost surely.
\end{lemma}
The proof of agreement is based on following idea: suppose that $F_{i\infty}(u)$ and $F_{j\infty}(u)$ disagree with positive probability for some $u\in \Q^d$. Then, by Lemma \ref{lem:avghelps}, there exists an average $\la F_{i\infty}(u)+(1-\la)F_{j\infty}(u)$ for some $0<\la<1$, which \emph{strictly} improve the mean squared error of $F_{i\infty}(u)$ for estimating $\one(\theta\leq u)$. On the other hand, after sufficiently long time $T$, agent $i$ can learn the empirical average $\frac{1}{T}\sum_{t=0}^{T-1}\one(Y_{j}(t)\leq u)$, which approximates $F_{j\infty}(u)$ in arbitrary precision because of the conditionally independent sampling mechanism and the convergence of $\nu_j(t)$ to $\nu_j(\infty)$. However, $F_{i\infty}(u)$ is the
posterior mean of $\one(\theta\leq u)$ under $\HH_i(\infty)$ and therefore must
be the best (i.e. lowest mean squared error) estimate of $\one(\theta\leq u)$, which 
 yields a contradiction. The full detail is given below.
\paragraph{Proof of Theorem \ref{thm:agreement}}
For simplicity, abbreviate $F_{i}(u)\equiv F_{i\infty}(u)=\P\left(\theta \leq u \given \HH_i(\infty)\right)$. Then, it suffices to show that for $u\in \Q^{d}$, $F_{i}(u)=F_{j}(u)$ almost surely. Having Lemma \ref{lem:avghelps} in mind, denote 
$Z_1:=F_i(u)-\one\left(\theta \leq u\right)\textnormal{ and }Z_2:=F_j(u)-\one\left(\theta \leq u\right)$,
which are bounded random variables. We now aim to show $\min_{\la\in [0,1]}\E\left(\la Z_1+(1-\la)Z_2\right)^2\geq \E Z_1^2 $. Note that if we establish this, then we have that by symmetry $\min_{\la\in [0,1]}\E\left(\la Z_1+(1-\la)Z_2\right)^2\geq \E Z_2^2$ holds, so by Lemma \ref{lem:avghelps}, $Z_1=Z_2$ almost surely. To this end, assume by contradiction that there exists some $\la\in[0,1]$ and $\eps>0$ such that
\begin{equation}\label{eq:contradiction}
\E\left[\left(\la F_{i}(u)+(1-\la)F_j(u)-\one(\theta \leq u)\right)^2\right]<\E\left[\left(F_{i}(u)-\one(\theta\leq u)\right)^2\right]-\eps. 
\end{equation}
Then, we argue that we can find a $\HH_{i}(\infty)$-measurable estimator of $\one(\theta \leq u)$ such that it is unbiased and has mean squared error smaller than $\E\left[\Var\left(\one(\theta\leq u)\given \HH_i(\infty)\right)\right]=\E\left(F_{i}(u)-\one(\theta\leq u)\right)^2$, which is a contradiction. Indeed, for $T\geq 1$, define
\begin{equation}\label{eq:def:Z1}
    Z_1(T):=\la F_i(u)+\frac{1-\la}{T}\sum_{t=1}^{T}\one\left(Y_{j}(t)\leq u\right).
\end{equation}
Then, because agent $i$ receives from agent $j$ the information $Y_j(t)$ at time $t$, $Z_1(T)$ is $\HH_i(\infty)$-measurable. Moreover, it is unbiased since by tower property, 
\begin{equation*}
    \E\left[Z_1(T)\right]
    =\la \P(\theta \leq u)+\frac{1-\la}{T}\sum_{t=1}^{T}\E\Big[\P\Big(Y_j(t)\leq u\given \HH_j(t)\Big)\Big]=\P(\theta \leq u).
\end{equation*}
Thus, 
\begin{equation}\label{eq:min:var}
    \E\Big[\big(Z_1(T)-\one(\theta\leq u)\big)^2\Big]\geq \E\left[\left(F_{i}(u)-\one(\theta\leq u)\right)^2\right].
\end{equation}
We now show that \eqref{eq:min:var} contradicts \eqref{eq:contradiction} for large enough $T$: note that
\begin{equation*}
    Z_1(T)=\la F_i(u)+(1-\la)F_j(u)+\frac{1-\la}{T}\sum_{i=1}^{T}\Big(\one\left(Y_{j}(t)\leq u\right)-F_{jt}(u)\Big)+\frac{1-\la}{T}\sum_{i=1}^{T}\Big(F_{jt}(u)-F_j(u)\Big).
\end{equation*}
Thus, Minkowski's inequality shows that $\E\Big[\big(Z_1(T)-\one(\theta\leq u)\big)^2\Big]^{1/2}\leq \Xi_1+\Xi_2+\Xi_3$, where
\begin{equation*}
\begin{split}
\Xi_1:&=\E\left[\left(\la F_{i}(u)+(1-\la)F_j(u)-\one(\theta \leq u)\right)^2\right]^{1/2}\\
\Xi_2:&=\frac{1-\la}{T}\E\left[\left(\sum_{i=1}^{T}\Big(\one\left(Y_{j}(t)\leq u\right)-F_{jt}(u)\Big)\right)^2\right]^{1/2}\\
\Xi_3:&=\frac{1-\la}{T}\E\left[\left(\sum_{i=1}^{T}\Big(F_{jt}(u)-F_j(u)\Big)\right)^2\right]^{1/2}.
\end{split}
\end{equation*}
Recall that $\Xi_1\leq \left(\E\left(F_{i}(u)-\one(\theta\leq u)\right)^2-\eps\right)^{1/2}$ by \eqref{eq:contradiction}. To this end, we aim to bound $\Xi_2$ and $\Xi_3$. To start with $\Xi_2$, note that $\{\one\left(Y_{j}(t)\leq u\right)-F_{jt}(u)\}_{t\geq 0}$ are uncorrelated at different times by the sampling mechanism of the agents. That is, for $0\leq t_1<t_2$, we have by tower property that
\begin{equation*}
\begin{split}
    &\E\Big[\Big(\one\left(Y_{j}(t_1)\leq u\right)-F_{jt_1}(u)\Big)\Big(\one\left(Y_{j}(t_2)\leq u\right)-F_{jt_2}(u)\Big)\Big]\\
    &=\E\bigg[\E\Big[\Big(\one\left(Y_{j}(t_1)\leq u\right)-F_{jt_1}(u)\Big)\Big(\one\left(Y_{j}(t_2)\leq u\right)-F_{jt_2}(u)\Big)\given \HH_j(t_2)\Big]\bigg]=0.
\end{split}
\end{equation*}
Thus, we can calculate and bound $\Xi_2$ by
\begin{equation}\label{eq:bound:B}
\Xi_2=\frac{1-\la}{T}\E\left[\sum_{t=1}^{T}\Big(\one\left(Y_{j}(t)\leq u\right)-F_{jt}(u)\Big)^2\right]^{1/2}\leq\frac{1}{\sqrt{T}}.
\end{equation}
For $\Xi_3$, recall that by L\'{e}vy's upward theorem that $\lim_{t\to\infty}F_{jt}(u)=F_{j}(u)$ almost surely. Since $F_{jt}(u), t\geq 0$ are bounded, dominated convergence theorem shows $\lim_{t\to\infty}\E\big((F_{jt}(u)-F_{j}(u)\big)^2=0$. Thus, for every $\delta>0$, there exists $T_0(\delta)$ such that if $t\geq T_0(\delta)$, then $\E\big((F_{jt}(u)-F_{j}(u)\big)^2<\delta$. Thus, Minkowski's inequality shows
\begin{equation}\label{eq:bound:C}
\begin{split}
    \Xi_3&\leq\frac{1-\la}{T}\E\left[\left(\sum_{i=1}^{T_0}\Big(F_{jt}(u)-F_j(u)\Big)\right)^2\right]^{1/2}+\frac{1-\la}{T}\sum_{t=T_0+1}^{T}\E\left[\left(F_{jt}(u)-F_j(u)\right)^2\right]^{1/2}\\
    &\leq \frac{T_0(\delta)}{T}+\delta^{1/2}.
\end{split}
\end{equation}
Recalling the bound $\E\Big[\big(Z_1(T)-\one(\theta\leq u)\big)^2\Big]^{1/2}\leq \Xi_1+\Xi_2+\Xi_3$, \eqref{eq:contradiction}, \eqref{eq:bound:B}, and  \eqref{eq:bound:C} show that
\begin{equation*}
    \E\Big[\big(Z_1(T)-\one(\theta\leq u)\big)^2\Big]^{1/2}\leq \left(\E\left[\left(F_{i}(u)-\one(\theta\leq u)\right)^2\right]-\eps\right)^{1/2}+\frac{1}{\sqrt{T}}+\frac{T_0(\delta)}{T}+\delta^{1/2}.
\end{equation*}
Thus, first taking $\delta>0$ sufficiently small compared to $\eps$ and then taking $T$ to be large enough, the equation above contradicts \eqref{eq:min:var}. This concludes the proof of Theorem \ref{thm:agreement}.

\subsection{Proof of Theorem \ref{thm:eff:learning}}
\label{subsec:proof:eff:learning}
The proof of Theorem \ref{thm:eff:learning} is 
based on a proof for the signal case in~\cite{MoSlTa:11}.
The paper ~\cite{MoSlTa:11} deals with The Economics signal model and it assumes a binary state. Here we are dealing with a sampling model.
Another major 
difference is that we have to deal with arbitrary prior distribution $\nu$, and in particular $\nu$ can have \emph{unbounded} support. To deal with this issue, we consider log-likelihood ratio for estimating $\one(\theta\in A)$, and apply Bayes rule as suggested in Definition \ref{def:bayes}.

\paragraph{Proof of Theorem \ref{thm:eff:learning}}
We claim that for every Borel measurable $A\subseteq \R^{d}$, $\P\left(\theta \in A\given \HH_i(\infty)\right)=\P(\theta\in A\given S_1,..,S_n)$ holds almost surely, which obviously achieves our goal by considering sets of the form $A=\prod_{i=1}^{d}(-\infty, q_i]$ for $q=(q_1,..,q_d)\in \Q^{d}$. To this end, fix $A$ such that $\P(\theta\in A) \in (0,1)$ since otherwise, our claim trivially holds. First, observe that using Bayes law \eqref{eq:bayeslaw} twice for $A$ and $A^{c}$ and dividing the two equations to eliminate $f$ gives
\begin{equation}\label{eq:crucial}
\frac{\P(\theta\in A\given S_1,...,S_n)}{\P(\theta\in A^{c}\given S_1,...,S_n)}=\prod_{i=1}^{n}\frac{\P(\theta\in A\given S_i)}{\P(\theta\in A^c\given S_i)}\cdot\left(\frac{\P(\theta\in A^c)}{\P(\theta\in A)}\right)^{n-1}.
\end{equation}
Note that in \eqref{eq:crucial}, the denominators may be zero which is problematic. However, our assumption that the signal structure has finite information, i.e. \eqref{eq:finite:info}, assures that $\frac{\P(\theta\in A\given S_i)}{\P(\theta\in A^c\given S_i)}\in (0,\infty)$ holds a.s.. Note that this also implies that $\P(\theta\in A^{c}\given S_1,...,S_n)>0$ almost surely because $f(S_1,...,S_n)>0$ holds in \eqref{eq:bayeslaw}. To this end, let
\begin{equation*}
    Z:= \sum_{i=1}^{n} Z_i\quad\textnormal{where}\quad Z_i :=\log\frac{\P(\theta\in A\given S_i)}{\P(\theta\in A^c\given S_i)}\in (-\infty,+\infty).
\end{equation*}
Thus, by \eqref{eq:crucial}, we have
\begin{equation}\label{eq:logit}
    W:=\P(\theta\in A\given S_1,...,S_n)= \frac{\beta e^Z}{1+\beta e^Z},\quad\textnormal{where}\quad \beta=\left(\frac{\P(\theta\in A^c)}{\P(\theta\in A)}\right)^{n-1}>0.
\end{equation}
Now, denote $X:=\P\left(\theta \in A\given \HH_i(\infty)\right)$. Then, by tower property, we have
\begin{equation*}
    X=\E\left[\P\Big(\theta\in A\given S_1,...,S_n,\HH_i(\infty)\Big)\bgiven \HH_i(\infty)\right]=\E\left[\P\Big(\theta\in A\given S_1,...,S_n\Big)\bgiven \HH_i(\infty)\right],
\end{equation*}
where the second equality holds because conditional on the private signals $S_1,..,S_n$, further conditioning on $\HH_i(\infty)$ does not change the posterior. Thus, we have
\begin{equation}\label{eq:X:W}
 X=\E\left[W \given \HH_i(\infty)\right]=\E\left[W \given X\right],
\end{equation}
where the second inequality is due to tower property. Note that $Z_i$ is $\HH_i(\infty)$-measurable and our assumption \eqref{eq:finite:info} guarantees that $Z_i$ is integrable for $i\leq n$. Thus, we have by tower property,
\begin{equation*}
    \E\left[Z_i W\given X\right]=\E\left[\E\left[Z_i W\given \HH_i(\infty)\right]\bgiven X\right]=\E\left[Z_i \E\left[W\given \HH_i(\infty)\right]\bgiven X\right]=\E\left[Z_i\given X\right]\cdot\E\left[W\given X\right],
\end{equation*}
where the last equality holds because of \eqref{eq:X:W}. Summing the equation above for $i=1,...,n$ and plugging in the expression \eqref{eq:logit} for $W$ gives
\begin{equation}\label{eq:chebyshev}
    \E\left[Z\cdot\frac{\beta e^{Z}}{1+\beta e^Z}\;\bigg|\; X\right]=\E\left[Z\given X\right]\cdot \E\left[\cdot\frac{\beta e^{Z}}{1+\beta e^Z}\;\bigg|\; X\right].
\end{equation}
Now, observe that $x \to \frac{\beta e^x}{1+\beta e^x}$ is strictly increasing. Thus, \eqref{eq:chebyshev} gives the equality condition for Chebyshev's sum inequality, which implies $Z=Z(X)$ is $\sigma(X)$-measurable. Thus, $W=\frac{\beta e^{Z}}{1+\beta e^Z}$ is $\sigma(X)$-measurable and by \eqref{eq:X:W}, $X=W$ holds a.s.. Therefore, $\P(\theta\in A\given S_1,...,S_n)=\P\left(\theta\in A\given \HH_i(\infty)\right)$ holds a.s., which concludes the proof of Theorem \ref{thm:eff:learning}.

\subsection{Proof of Theorem \ref{thm:bernoulli:sampling} and Proposition \ref{prop:misspecified:bernoulli}}
\label{subsec:proof:misspecified}
In this section, we prove Theorem \ref{thm:bernoulli:sampling} and Proposition \ref{prop:misspecified:bernoulli}. Although Theorem \ref{thm:bernoulli:sampling} is based on similar ideas used for the proof of Theorems \ref{thm:agreement} and \ref{thm:eff:learning}, Proposition \ref{prop:misspecified:bernoulli} requires a new idea. 

First, we show that $\wt{p}(\theta\given z_1\ldots, z_m)$, defined in \eqref{eq:posterior:beta}, is the beta distribution.
\begin{lemma}\label{lem:beta}
Let $\theta\in \Unif([0,1])$ and given $\theta$, suppose $Z_1,\ldots,Z_m$ are i.i.d. $\Ber(\theta)$. Then, the conditional distribution of $\theta$ given $Z_1,\ldots,Z_n$ is $\textnormal{Beta}\left(\sum_{i=1}^{m}Z_i+1,m+1-\sum_{i=1}^{m}Z_i\right)$. In particular, we have
\begin{equation}\label{eq:beta}
\E[\theta\given Z_1,\ldots, Z_m]=\frac{\sum_{i=1}^{m}Z_i+1}{m+2}\,.
\end{equation}
\end{lemma}
\begin{proof}
By Bayes rule, we have
\[
p(\theta\given Z_1,\ldots, Z_m)\propto \prod_{i=1}^{m}p(Z_i\given \theta)\cdot p(\theta)=\theta^{\sum_{i=1}^{m}Z_i}(1-\theta)^{m-\sum_{i=1}^{m}Z_i}\one\big(\theta\in [0,1]\big)\,.
\]
Thus, the conditional distribution of $\theta$ given $Z_1,\ldots,Z_n$ is $\textnormal{Beta}\left(\sum_{i=1}^{m}Z_i+1,m+1-\sum_{i=1}^{m}Z_i\right)$. Then, it is well-known that for $X\sim \textnormal{Beta}(\alpha,\beta)$, $\E X=\alpha/(\alpha+\beta)$ holds, thus \eqref{eq:beta} follows.
\end{proof}

\paragraph{Proof of Theorem \ref{thm:bernoulli:sampling}}
We first prove the a.s. and $L^1$ convergence of $m_i(t)$. Since $\theta$ is supported on $[0,1]$ and $\{\KK_i(t)\}_{t\geq 1}$ is increasing $\sigma$-algebra, $\{m_i(t)\}_{t\geq 1}$ is a Doob martingale. Thus, if we denote $\KK_i(\infty):=\cup_{t\geq 1}\KK_i(t)$, then by L\'{e}vy's upward theorem (see e.g.~\cite[Theorem 4.6.8]{Durrett19}), $m_i(t)\to m_i(\infty):=\E[\theta\given \KK_i(\infty)]$, a.s. and in $L^1$.

Next, we prove that $m_i(\infty)= m_j(\infty)$, a.s. if $i$ and $j$ are neighbors. We follow a similar strategy as in the proof of Theorem \ref{thm:agreement}. Suppose by contradiction that $\P\big(m_i(\infty)\neq m_j(\infty)\big)>0$. Without loss of generality, assume that $\E(m_i(t)-1/2)^2\leq \E(m_j(t)-1/2)^2$. Then, Lemma \ref{lem:avghelps} applied to $Z_1=m_i(t)-1/2$ and $Z_2=m_j(t)-1/2$ shows that there exists some $\lambda\in [0,1]$ and $\eps>0$ such that
\begin{equation}\label{eq:X:contradiction}
\E\left[\left(\la m_i(t)+(1-\la)m_j(t)-\frac{1}{2}\right)^2\right]<\E\left[\left(m_i(t)-\frac{1}{2}\right)^2\right]-\eps\,.
\end{equation}
Now, analogously to \eqref{eq:def:Z1}, consider the estimator
\[
\overline{Z}_1(T)= \lambda m_i(t)+\frac{1-\lambda}{T}\sum_{t=1}^{T} X_j(t)\,,
\]
which is $\KK_i(\infty)$-measurable since $j$ is a neighbor of $i$. Recall that given $Y_j(t)$ and $\KK_j(t)$, $X_j(t)\sim \Ber(Y_j(t))$, and that $Y_j(t)\sim \law\big((\theta\given \KK_j(t)\big)$. Thus, tower property gives
\begin{equation}\label{eq:X:j:unbiased}
\E[X_j(t)\given \KK_j(t)]= \E\Big[\E[X_j(t)\given Y_j(t),\KK_j(t)]\bgiven \KK_j(t)\Big]=\E[Y_j(t)\given \KK_j(t)]=m_j(t)\,.
\end{equation}
Thus, for $0\leq t_1<t_2$, we have by tower property that
\begin{equation}\label{eq:X:uncorrelated}
\E\Big[\left(X_j(t_1)-m_j(t_1)\right)\left(X_j(t_2)-m_j(t_2)\right)\Big]=\E\left[\E\left[\left(X_j(t_1)-m_j(t_1)\right)\left(X_j(t_2)-m_j(t_2)\right)\bgiven \KK_j(t_2)\right]\right]=0\,.
\end{equation}
Having established \eqref{eq:X:j:unbiased} and \eqref{eq:X:uncorrelated}, the same argument as in the proof of Theorem \ref{thm:agreement} implies the following: by our assumption \eqref{eq:X:contradiction}, for large enough $T\geq T(\eps)$, we have
\begin{equation}\label{eq:contradict:conclusion}
\E\left[\left(\overline{Z}_1(T)-\frac{1}{2}\right)^2\right]<\E\left[\left(m_i(t)-\frac{1}{2}\right)^2\right]\,.
\end{equation}
On the other hand, since $\overline{Z}_1(T)$ is $\HH_i(\infty)$-measurable unbiased estimator of $\theta\sim \Unif([0,1])$ by \eqref{eq:X:j:unbiased}, thus \eqref{eq:contradict:conclusion} gives a contradiction. Therefore, we must have that $m_i(\infty)= m_j(\infty)$ a.s. for $i \sim j$, which implies that $m_i(\infty)\equiv m(\infty)$ for every $i\in V$ since $G$ is connected.

Finally, we prove that the agreed limiting posterior mean $m(\infty)=m_i(\infty), i\in V,$ equals $\E[\theta\given S_1,\ldots, S_n]$. By Lemma \ref{lem:beta}, we have that
\[
\overline{W}:=\E[\theta \given S_1,\ldots, S_n]=\frac{\sum_{i=1}^{n}S_i+1}{n+2}.
\]
Now, note that since $m(\infty)=\E[\theta\given \KK_i(\infty)]$ holds for $i\in V$, we have
\[
m(\infty)=\E\left[\E[\theta\given S_1,\ldots, S_n,\KK_i(\infty)]\bgiven \KK_i(\infty)\right]=\E\left[\E[\theta\given S_1,\ldots, S_n]\bgiven \KK_i(\infty)\right]=\E\big[\overline{W}\sbgiven \KK_i(\infty)\big]\,,
\]
where the second equality holds because conditional on the private signals $S_1,\ldots, S_n$, further conditioning on $\KK_i(\infty)$ does not change the posterior mean. Thus, it follows that
\begin{equation}\label{eq:optimal:crucial}
m(\infty)=\E\big[\overline{W}\sbgiven \KK_i(\infty)\big]=\E\big[\overline{W}\sbgiven m(\infty)\big]\,.
\end{equation}
Note that since $S_i$ is $\KK_i(\infty)$-measurable, we have
\[
\E\big[S_i \overline{W} \sbgiven m(\infty)\big]=\E\left[\E\big[S_i\overline{W}\sbgiven \KK_i(\infty)]\bgiven m(\infty)\right]=\E\left[S_i\E\big[\overline{W}\sbgiven \KK_i(\infty)]\bgiven m(\infty)\right]=\E\big[S_i\sbgiven m(\infty)]\cdot \E\big[\overline{W}\sbgiven m(\infty)]\,.
\]
Thus, summing the above equality for $i=1,\ldots, n$ shows that
\[
\E\left[(\overline{W})^2\bgiven m(\infty)\right]=\left(\E\left[ \overline{W}\sbgiven m(\infty)\right]\right)^2\,,
\]
which implies that $\overline{W}$ is $m(\infty)$-measurable. Therefore, combining with \eqref{eq:optimal:crucial}, we have 
\[
m(\infty)=\overline{W}\equiv \E[\theta\given S_1,\ldots, S_n]\,,
\]
almost surely, which concludes the proof of Theorem \ref{thm:bernoulli:sampling}.
\paragraph{Proof of Proposition \ref{prop:misspecified:bernoulli}}
By Lemma \ref{lem:beta}, note that the posteriors $\wt{\nu}_i(t)$ are beta distributions for $i=1,2$. To this end, denote their mean by
\[
M_1(t):=\frac{S_1+\sum_{s=1}^{t-1}X_2(s)+1}{t+2}\,,\quad M_2(t):=\frac{S_2+\sum_{s=1}^{t-1}X_1(s)+1}{t+2}\,.
\]
Then, Lemma \ref{lem:beta} shows that
\begin{equation}\label{eq:nu:beta}
\wt{\nu}_1(t)=\Beta\big((t+2)M_1(t)\,,\,(t+2)(1-M_1(t)\big)\,,\quad \wt{\nu}_2(t)=\Beta\big((t+2)M_2(t)\,,\, (t+2)(1-M_2(t)\big)\,.
\end{equation}
First, we claim that $M(t):=\frac{M_1(t)+M_2(t)}{2}$ converges almost surely to $M(\infty)$ by showing that $(M(t))_{t\geq 1}$ is a martingale. To this end, denote the $\sigma$-algebra $\FF(t):= \sigma\big(S_1,S_2, (X_i(s))_{i\in \{1,2\}, 1\leq s\leq t-1}\big)$. Then, $(M(t))_{t\geq 1}$ is $\FF(t)$-adapted. Moreover, we have by tower property that
\begin{equation}\label{eq:conditional:X:M}
\E\big[X_1(t)\sbgiven \FF(t)]=\E\Big[E\big[X_1(t)\sbgiven \FF(t), Y_1(t)\big]\bgiven \FF(t)\Big]=E\big[Y_1(t)\sbgiven \FF(t)]=M_1(t)\,,
\end{equation}
where the last equality holds since conditional on $\FF(t)$, $Y_2(t)\sim \wt{\nu}_2(t)$ and the mean of $\wt{\nu}_2(t)$ is $M_2(t)$ by \eqref{eq:nu:beta}. By symmetry, $\E[X_1(t)\given \FF(t)]=M_1(t)$ also holds. Thus, it follows that
\begin{equation}\label{eq:conditional:M}
\begin{split}
&\E\big[M_1(t+1)\sbgiven \FF(t)]
=\frac{t+2}{t+3}M_1(t)+\frac{1}{t+3}M_2(t)\,,\\
&\E\big[M_2(t+1)\sbgiven \FF(t)]
=\frac{t+2}{t+3}M_2(t)+\frac{1}{t+3}M_1(t)\,.
\end{split}
\end{equation}
Hence, adding the two equations above implies $\E[M(t+1)\given \FF(t)]=M(t)$. Further, $M(t)\in [0,1]$ is bounded,  thus $(M(t), \FF(t))_{t\geq 1}$ is a bounded martingale. Therefore, Doob's martingale convergence theorem shows that $M(t)$ converges almost surely to a random variable $M(\infty)\in [0,1]$, which proves our first claim.

Second, we claim that $D(t):=M_1(t)-M_2(t)\to 0$ as $t\to\infty$, almost surely. The proof is motivated from the analysis of Bernard Friedman's urn model in \cite{Freedman65}. First, observe that
\begin{equation}\label{eq:D:condition}
\begin{split}
\E\big[D(t+1)^2\sbgiven \FF(t)]
&=\frac{1}{(t+3)^2}\E\Big[\big((t+2)D(t)+X_2(t)-X_1(t)\big)^2\bgiven \FF(t)\Big]\\
&=\frac{t(t+2)}{(t+3)^2}D(t)^2+\frac{1}{(t+3)^2}\E\Big[\big(X_2(t)-X_1(t)\big)^2\bgiven \FF(t)\Big] \,,
\end{split}
\end{equation}
where the last equality holds since $\E[X_2(t)-X_1(t)\given \FF(t)]=-D(t)$ holds (cf. \eqref{eq:conditional:X:M}). Observe that since $X_1(t),X_2(t)\in \{0,1\}$, the last term satisfies $\E\big[\big(X_2(t)-X_1(t)\big)^2\bgiven \FF(t)\big]\leq 1$, a.s.. Thus, taking expectation in the equation above, we can bound
\begin{equation}\label{eq:important:ineq}
a_{t+1}:= \E\big[D(t+1)^2\big]\leq \frac{t(t+2)}{(t+3)^2}a_t+\frac{1}{(t+3)^2}\,.
\end{equation}
The important observation in \eqref{eq:important:ineq} is that $\frac{1}{(t+3)^2}$ is summable while $\frac{t(t+2)}{(t+3)^2}=1-O(t^{-1})$ holds. Any non-negative sequence $(a_t)_{t\geq 0}$ satisfying \eqref{eq:important:ineq} must converge to $0$: note that $(t+3)a_{t+1}\leq (t+2)a_t+\frac{1}{t+3}$ holds, and iterating gives
\[
(t+2)a_t\leq 2a_0+\sum_{s=0}^{t-1}\frac{1}{s+3}\asymp \log t\,,
\]
thus $a_t\lesssim \frac{\log t}{t}\to 0$ as $t\to\infty$. Therefore, we have shown that $D_t\to 0$ in $L^2$. Finally, note that if we denote $A(t):= D(t)^2-\sum_{i=0}^{t}\frac{1}{(i+2)^2}$, which is $\FF(t)$-adapted and bounded below, then \eqref{eq:D:condition} implies that
\[
\begin{split}
\E\big[A(t+1)\sbgiven \FF(t)\big]
&=\frac{t(t+2)}{(t+3)^2}D(t)^2+\frac{1}{(t+3)^2}\E\Big[\big(X_2(t)-X_1(t)\big)^2\bgiven \FF(t)\Big]-\sum_{i=0}^{t+1}\frac{1}{(i+2)^2}\\
&\leq D(t)^2-\sum_{i=0}^{t}\frac{1}{(i+2)^2}=A(t)\,,
\end{split}
\]
thus $\big(A(t)\big)_{t\geq 0}$ is a super martingale bounded below. Hence, by Doob's martingale convergence theorem, $A(t)\to A(\infty)$, a.s. for a random variable $A(\infty)$. This further implies that 
\[
D(t)\to D(\infty):= A(\infty)+\sum_{i=0}^{\infty}\frac{1}{(i+2)^2}\,.
\]
Since $D(t)\in [-1,1]$ is bounded and we have already shown that $a_t=\EE[D(t)^2]\to 0$, we must have that $D(\infty)=0$, a.s.. Therefore, we have proven the second claim that $D(t)\to 0$, almost surely.

Consequently, the previous two claims imply that $M_i(t)\to M(\infty)$ a.s. for $i=1,2$. Note that this implies that almost surely, the random measure $\wt{\nu}_i(t)$ converges to $\delta_{M(\infty)}$. Indeed, if we respectively denote $\E_{\wt{\nu}_i(t)}[X]$ and $\Var_{\wt{\nu}_i(t)}(X)$ by the expectation and the variance of $X\sim \wt{\nu}_i(t)$, then by \eqref{eq:nu:beta},
\[
\E_{\wt{\nu}_i(t)}[X]=M_i(t)\to M(\infty)\quad\textnormal{and}\quad \Var_{\wt{\nu}_i(t)}(X)=\frac{M_i(t)(1-M_i(t))}{t+3}\to 0\,,\quad\textnormal{a.s..}
\]
Furthermore, having established that $M_i(t)\to M(\infty)$ a.s. for $i=1,2$, it is standard in the analysis of urn models that $M(\infty)$ cannot have a point mass conditional on the initial data $S_1,S_2$. That is, the distribution of $M(\infty)$ conditional on $S_1,S_2$ cannot have an atom whose proof is the same with the analysis for time-dependent P\'{o}lya urn in \cite[Theorem 3]{Pemantle90} (see also \cite[Theorem 1.1]{AnMoRa:16}). Therefore, it follows that $\P\big(M(\infty)\neq \E[\theta\given S_1,\ldots, S_n]\big)=1$ holds, which concludes the proof of Proposition \ref{prop:misspecified:bernoulli}.

\section{Further discussions and future research directions}
In this section, we conclude by further discussing the importance of our main results, and mentioning open problems. In Section \ref{subsec:comparison}, we compare our main results for the sampling model in Section \ref{subsec:model} with the previous results for the action model. In Section \ref{subsec:further}, we discuss other aspects of the sampling model. In Section \ref{subsec:open}, we list open problems for future research directions.

\subsection{A comparison of the sampling and action models}\label{subsec:comparison}
With our main results in Section \ref{sec:model}, we can compare the statistical efficiency of a sampling model with that of an action model, where agents can only communicate binary information. Consider the simplest nontrivial graph $G$ consisting of single edge and the unknown quantity $\theta$ is in $\{0,1\}$ with the uniform prior, i.e. $\nu(0)=\nu(1)=1/2$. Further assume that the signals $S_i, i=1,2,$ has the density $2x$ in $[0,1]$ (w.r.t. to the Lebesgue measure) if $\theta=1$ and the density $2-2x$ in $[0,1]$ if $\theta=0$. For simplicity, we consider the action model where individuals are very limited and can only communicate $0$ or $1$ to their neighbors are at each round and that they do so simultaneously.
\begin{enumerate}
    \item In the action model, assuming a convex loss function that reveals the more likely posterior, agents communicate the more likely state of the world at each round. 
    In this case it is known that the agents will eventually agree in their views on the more likely state. However, it is also known that generally they will not agree on what the posterior is and their posteriors will not agree with the true posterior given all the data~~\cite{GaleKariv:03,RoSoVi:09}. 
    To see this in our simple example, consider the case where the first agent receives a signal $x_1 > 0.5$, and the second agent receives a signal $x_2 > 0.5$. In this case it is easy to verify that both agents will communicate $1$ indefinitely, agent $1$ posterior probability for $1$ will be $6 x_1 / (2+4x_1)$, agent $2$ posterior will be $6 x_2 / (2+4x_2)$. 
    This follows since $\P(x_1 \given \theta = 1)/\P(x_1 \given \theta = 0) = 2 x_1 /(2 - 2 x_1), \P(x_2 > 0 \given \theta = 1)/\P(x_2 > 1 \given \theta = 0) = 3$ so by Bayes law: 
    \[
    \frac{\P(\theta = 1\given x_1, x_2 >1/2)}{\P(\theta = 0\given| x_1, x_2 >1/2)} = 
     \frac{\P(x_1, x_2 >1/2 \given \theta = 1)}
    {\P(x_1, x_2 >1/2 \given \theta = 0)} = \frac{6 x_1}{2 - 2x_1}
    \]
    while the true posterior is $x_1 x_2 / (x_1 x_2 + (1-x_1) (1-x_2))$. 
    \item In the sampling model each individual communicates $1$ with probability $p$, where $p$ is its current posterior.
    Since the conditions of Theorem \ref{thm:eff:learning} is satisfied (see Example \ref{eg:lr:bounded}), Theorems \ref{thm:agreement} and \ref{thm:eff:learning} imply that both agents will have the same posterior distribution in the limit, which is the exact posterior given the original signals of both agents, i.e. $x_1 x_2 / (x_1 x_2 + (1-x_1) (1-x_2))$. 
\end{enumerate}

Thus, as a learning mechanism, sampling models are statistically more efficient than action models even in this simplest case and the phenomena in the example can be generalized to other networks, signal structures, and actions as long as the action space is finite. 

\subsection{Further discussions on the sampling model}
\label{subsec:further}
A common critique of the results on action and belief models on networks is that these results assume all individuals know the exact structure of the network. 
Our models make the same assumption. However, unlike many of the results in learning on networks, our results are of interest for any network, no matter how simple. Even in the case where the graph is a single edge, the convergence properties of the sampling model is highly non-trivial, and we will have to leave many questions open, e.g. speed of convergence. In this case, the assumption that the individuals know the network structure is perfectly reasonable. This assumption may also be reasonable for other small simple networks (e.g. a path of length $2$). Since the proofs apply for all connected networks, they are given in the more general setting. 

We are not aware of previous theoretical work considering sampling model on networks. In the majority of works, there is a single individual who interacts with the environment (and this is repeated many times). 
Our result analyze the asymptotic behavior of a sampling model involving {\em multiple} individuals.
In prior work, there are some sophisticated works involving the sampling model with a single individual. 
For example, the result of~\cite{SanbornGriffith:07} analyze a sampling model involving a single individual, which converges to the prior distribution. 
In~\cite{SanbornGriffith:07}, individuals do not change their distribution during the process. Instead the process is used to learn what is the prior distribution.

We emphasize that in this work, we
assume that one cannot design the actions of the agents. Rather, we view the sampling model as given to us, which models how individuals learn from each other. Our results show that despite of the
complexity of the sampling model, it is statistically optimal.

\subsection{Open Problems}\label{subsec:open}

Our work naturally leads to a number of fascinating future research directions:
\begin{itemize}
    \item It is natural to ask what is the convergence rate of the sampling models studied here. For example, how does the speed of convergence depend on the topology of the graph $G$? This is a very interesting question that is currently out of reach: in the Gaussian model (cf. Example \ref{eg:gaussian}), we compute the posterior \textit{exactly} in Appendix \ref{sec:gaussian}. Even for the simplest case where $G$ is a single edge, the convergence rate of the recursion (see Lemma \ref{lem:gaussian}) is highly non-trivial. We report the numerical simulation in Appendix \ref{subsec:numerical}, and empirically observe that, as one might expect, when $G$ is the clique, the speed of converge is faster than when $G$ is a cycle.
    \item To what extent can the result be extended if instead of exact sampling from the posterior some approximate sampling is used?
    The study of approximate Bayesian inference includes many such approximations, see e.g.~\cite{GeHoTe:15} and it is fascinating to explore to what extent do our results carry over these models. 
    \item A large body of experimental research supports sampling models or ``probability matching". It is fascinating to design an experiment to test if sampling models are good model of human interaction of multiple individuals who communicate with each other.  
\end{itemize}
\bibliographystyle{amsalpha}
\bibliography{social_learning,all,my}
\appendix

   \section{Omitted proofs}\label{sec:omittedproofs}
   In this appendix, we provide the omitted proofs from the main text. In Section \ref{subsec:proof:convergence}, we prove the convergence of posteriors of agents in Proposition \ref{lem:genconvergence}. In Section \ref{subsec:proof:suff:avghelps}, we prove Lemma \ref{prop:suff} and Lemma \ref{lem:avghelps}.

\subsection{Proof of convergence of posteriors}
\label{subsec:proof:convergence}
\begin{prop}[Convergence]\label{lem:genconvergence}
The sequence of (random) probability measures $(\nu_i(t))_{t\geq 0}$ converges weakly to (random) law $\nu_i(\infty):=\law\left(\theta \given \HH_i(\infty)\right)$, almost surely, i.e. for every $i\in V$, we have
\begin{equation}
    \nu_i(t)\wto \nu_i(\infty), \quad\textnormal{almost surely.}
\end{equation}
\end{prop}
\begin{proof}
 We have already shown that \eqref{eq:convergence:distribution} happens. Now, assume that \eqref{eq:convergence:distribution} happens and let $x\in \R^{d}$ be the continuity point of $F_{i\infty}(\cdot)$. Then for any $\eps>0$, there exists $y_1,y_2\in \R^{d}$ such that $y_1<x<y_2$ holds (here, $a<b$ for $a,b\in \R^{d}$ denotes coordinate-wise $a_i<b_i, i\leq d$), and \begin{equation*}
    F_{i\infty}(x)-\eps <F_{i\infty}(y_1)\leq F_{i\infty}(y_2)< F_{i\infty}(x)+\eps.
\end{equation*}
Since $\Q^{d}$ is dense in $\R^{d}$ and $y_1<x<y_2$ holds, there exists $q_1,q_2\in \Q^{d}$ such that $y_1<q_1<x<q_2<y_2$. Since $F_{i\infty}(\cdot)$ is a distribution function, it is monotonically increasing in each coordinate. Thus, the above inequality shows
\begin{equation*}
    F_{i\infty}(x)-\eps <F_{i\infty}(q_1)\leq  F_{i\infty}(q_2)< F_{i\infty}(x)+\eps.
\end{equation*}
Note that by \eqref{eq:convergence:distribution}, we have $\big|F_{it}(q_1)-F_{i\infty}(q_1)\big|\vee \big|F_{it}(q_2)-F_{i\infty}(q_2)\big|<\eps$ for sufficiently large $t$. Moreover, $F_{it}(\cdot)$ is also monotonically increasing in each coordinate. Thus, we have
\begin{equation*}
    F_{i\infty}(x)-2\eps <F_{it}(q_1)\leq F_{it}(x)\leq F_{it}(q_2)< F_{i\infty}(x)+2\eps.
\end{equation*}
Since $\eps$ is arbitrary, we have shown that $F_{it}(x)\to F_{i\infty}(x)$ as $t\to\infty$ holds on the continuity points of $F_{i\infty}$, which concludes the proof.
\end{proof}

\subsection{Proof of Lemma \ref{prop:suff} and Lemma \ref{lem:avghelps}}
\label{subsec:proof:suff:avghelps}
We first prove Lemma \ref{prop:suff}. It is proved via $2$ separate lemmas, namely Lemmas \ref{lem:bayes} and \ref{lem:finite} below:
\begin{lemma}[A sufficient condition for Bayes law]\label{lem:bayes}
 Suppose that there exists $\sigma$-finite regular Borel measures $\left\{\mu_i\right\}_{i \in V}$ such that the $\P_i(\cdot\given \theta)$ is absolutely continuous with respect to $\mu_i$ for every $\theta \in \supp(\nu)$ and $i\in V$. Then, $\left\{\P_i(\cdot\given \theta)\right\}_{i\in V}$ satisfies the Bayes law in Definition \ref{def:bayes}.
\end{lemma}
\begin{proof}
Denote $\P_{S}$ by the joint distribution of $(S_1,...,S_n)$. Then, for $x=(x_1,...,x_n)\in \supp(\P_{S})$ and $\eps>0$, we have by conditional independence of $S_i$ given $\theta$,
\begin{equation*}
\begin{split}
    &\P\Big(\theta \in A \bgiven S_i\in B_{\eps}(x_i),\,\forall i\leq n\Big)\\
    &=\frac{\prod_{i=1}^{n}\P\Big( S_i \in B_{\eps}(x_i)\bgiven \theta \in A\Big)}{\P\Big(S_i\in B_{\eps}(x_i),\,\forall i\leq n \Big)}\cdot \P(\theta\in A)
    \\
    &=\prod_{i=1}^{n}\P\Big(\theta \in A \given S_i \in B_{\eps}(x_i)\Big) \cdot \frac{\prod_{i=1}^{n}\P\Big(S_i\in B_{\eps}(x_i)\Big)}{\P\Big(S_i\in B_{\eps}(x_i),\,\forall i\leq n \Big) }\cdot\P(\theta\in A)^{-n+1},
\end{split}
\end{equation*}
where $B_{\eps}(z)$ for $z\in \R^{d}$ denotes the $L^2$ ball with radius $\eps$ around $z$. We now argue that sending $\eps\to 0$ gives our goal \eqref{eq:bayeslaw}. First, we have
\begin{equation*}
    \lim_{\eps \to 0}\P\Big(\theta \in A \bgiven S_i\in B_{\eps}(x_i),\,\forall i\leq n\Big)=\P\Big(\theta \in A \bgiven S_i=x_i,\,\forall i \leq n\Big),\quad\P_{S}\textnormal{-a.s }x, 
\end{equation*}
because if we consider the measure $\nu_A(B):=\P(\theta\in A,S \in B)$ for Borel measurable $B\subseteq \R^{dn}$, then $\nu_A\ll \P_S$ holds, so (generalized) Lebesgue differentiation theorem on Euclidean space (see e.g. Theorem 8.4.6 of \cite{BeCz09}) shows that for $\P_S$-a.s. $x$, 
\begin{equation*}
    \lim_{\eps \to 0}\frac{\nu_A\Big(\prod_{i=1}^{n}B_{\eps}(x_i)\Big)}{\P_S\Big(\prod_{i=1}^{n}B_{\eps}(x_i)\Big)}=\frac{\mathrm{d}\nu_A}{\rd \P_S}(x)=\P \Big(\theta\in A \bgiven S_i=x_i,\,\forall i \leq n\Big), 
\end{equation*}
where the last equality holds by the definition of regular conditional probability distribution. By the same argument, we have for every $i\leq n$,
\begin{equation*}
  \lim_{\eps \to 0}\P\Big(\theta \in A \bgiven S_i\in B_{\eps}(x_i)\Big)=\P\Big(\theta \in A \bgiven S_i=x_i\Big),\quad\P_{S_i}\textnormal{-a.s }x_i.
\end{equation*}
Thus, by the $4$ equations in the display above, proof is done if we show $\lim_{\eps\to 0}\frac{\prod_{i=1}^{n}\P\Big(S_i\in B_{\eps}(x_i)\Big)}{\P\Big(S_i\in B_{\eps}(x_i),\,\forall i\leq n \Big) }$ exists in $(0,\infty)$ for $\P_S$-a.s. $x$. Here, we will use our assumption that $\P_i(\cdot\given \theta)\ll \mu_i$. Denote its Radon-Nikodym derivative by $f_i(x\given \theta):= \frac{\rd \P_i(\cdot\given\theta)}{\rd \mu_i}(x)$. Then, for any Borel measurable $B_i \subseteq \R^{d}$, we have by Toninelli thoerem,
\begin{equation*}
    \P(S_i \in B_i) =\int_{\R^{d}}\P(S_i\in B_i \given \theta) \nu(\rd \theta)=\int_{B_i}\int_{\R^{d}}f_i(x\given\theta)\nu(\rd \theta)\mu_i(\rd x).
\end{equation*}
Thus, $\P_{S_i}\ll \mu_i$ holds with Radon-Nikodym derivative $\frac{\rd \P_{S_i}}{\rd \mu_i}(x)=\int_{\R^{d}}f_i(x\given\theta)\nu(\rd \theta)$. In particular, we have $\int_{\R^{d}}f_i(x\given\theta)\nu(\rd \theta)\in (0,\infty)$ for $\P_{S_i}$-a.s. $x$. Moreover, again by Lebesgue differentiation theorem, we have
\begin{equation}\label{tech-1}
    \lim_{\eps \to 0}\frac{\P(S_i\in B_{\eps}(x_i))}{\mu_i\left(B_{\eps}(x_i)\right)}=\int_{\R^{d}}f_i(x_i\given\theta)\nu(\rd \theta),\quad\P_{S_i}\textnormal{-a.s }x_i.
\end{equation}
By conditional independence of $S_i$'s given $\theta$, the same argument gives that we have that $\P_S\ll \otimes_{i=1}^{n}\mu_i$ with Radon-Nikodym derivative $\frac{\rd \P_S}{\rd \otimes_{i=1}^{n}\mu_i}(x)=\int_{\R^{d}}\prod_{i=1}^{n}f_i(x_i\given \theta)\nu(\rd \theta)\in (0,\infty)$, $\P_S$-a.s. $x$. Lebesgue differentiation theorem again gives
\begin{equation}\label{tech-2}
    \lim_{\eps \to 0}\frac{\P\Big(S_i\in B_{\eps}(x_i),\,\forall i \leq n\Big)}{\prod_{i=1}^{n}\mu_i\left(B_{\eps}(x_i)\right)}=\int_{\R^{d}}\prod_{i=1}^{n} f_i(x_i\given\theta)\nu(\rd \theta),\quad\P_{S}\textnormal{-a.s }x.
\end{equation}
Therefore, by \eqref{tech-1} and \eqref{tech-2} we have
\begin{equation*}
    \lim_{\eps\to 0}\frac{\prod_{i=1}^{n}\P\Big(S_i\in B_{\eps}(x_i)\Big)}{\P\Big(S_i\in B_{\eps}(x_i),\,\forall i\leq n \Big) }=\frac{\prod_{i=1}^{n}\int_{\R^{d}} f_i(x_i\given\theta)\nu(\rd \theta)}{\int_{\R^{d}}\prod_{i=1}^{n} f_i(x_i\given\theta)\nu(\rd \theta)}\in (0,\infty)\quad\P_{S}\textnormal{-a.s }x, 
\end{equation*}
which concludes the proof.
\end{proof}

\begin{lemma}[A sufficient condition for finite information]\label{lem:finite}
For $i\in V$, let $\P_{S_i}$ denote the marginal law of $S_i$. Under the same assumptions as in Lemma \ref{lem:bayes}, $\P_{S_i}$-a.s., the posterior $\law(\theta\given S_i)$ is absolutely continuous with respect to the prior $\nu$. Denote $p_i(x\given S_i)$ by its Radon-Nikodym derivative. Then, if
\begin{equation*}
    \E_{S_i}\int_{\R}\Big|\log \left(p_i\left(x\given S_i\right)\right)\Big| \nu (dx) <\infty,
\end{equation*}
then $\P_i(\cdot\given \theta)$ has finite information in the sense of Definition \ref{def:finite}.
\end{lemma}
\begin{proof}
For the first claim, if there exists a $\sigma$-finite $\mu_i$ such that $\P_i(\cdot\given\theta)\ll \mu_i$ for every $\theta\in \supp(\nu)$, then classical Bayes' theorem (see for e.g. Theorem 1.31 of \cite{Schervish95}) shows that $\P_{S_i}$-a.s., $\law(\theta\given S_i)\ll \nu$.

For the second claim, take Borel measurable $A$ such that $0<\nu(A)<1$. First, note that
\begin{equation*}
    \E\left|\log\left(\frac{\P(\theta\in A\given S_i)}{\P(\theta\in A^{c}\given S_i)}\right)\right|\leq \E\Big [-\log\P\left(\theta\in A\given S_i\right)\Big]+\E\Big[-\log\P\left(\theta\in A^c\given S_i\right)\Big],
\end{equation*}
thus to prove $\P_i(\cdot\given \theta)$ has finite information, it suffices to show $\E\left[-\log\P\left(\theta\in A\given S_i\right)\right]<\infty$. We can express
\begin{equation*}
    \E\Big [-\log\P\left(\theta\in A\given S_i\right)\Big]=\E\left[-\log\left(\nu(A)^{-1}\int_{A}p_i\left(x\given S_i\right)\nu(dx)\right)\right]+\E\left[-\log \nu(A)\right]
\end{equation*}
Since $x\to -\log x$ is convex, we can use Jensen's inequality to bound the first piece in the RHS above:
\begin{equation*}
\begin{split}
    \E\left[-\log\left(\nu(A)^{-1}\int_{A}p_i\left(x\given S_i\right)\nu(dx)\right)\right]
    &\leq \E\left[\nu(A)^{-1}\int_{A}-\log\left(p_i\left(x\given S_i\right)\right)\nu(dx)\right]\\
    &\leq \nu(A)^{-1} \E \int_{\R}\Big|\log \left(p_i\left(x\given S_i\right)\right)\Big| \nu (dx)<\infty.
\end{split}
\end{equation*}
Thus, for $\nu(A)\in (0,1)$, $\E\left[-\log\P\left(\theta\in A\given S_i\right)\right]<\infty$ and $\E\left[-\log\P\left(\theta\in A^c\given S_i\right)\right]<\infty$ hold, which concludes the proof.
\end{proof}
\paragraph{Proof of Lemma \ref{prop:suff}}
This is immediate from Lemma \ref{lem:bayes} and Lemma \ref{lem:finite}.
\paragraph{Proof of Lemma \ref{lem:avghelps}}
We aim to show that $\E(Z_1 - Z_2)^2=0$.
Let $M=(M_{i,j})_{i,j\leq 2}$ denote the $2\times 2$ matrix encoding the second moment of the random vector $(Z_1, Z_2)$, i.e. $M_{ii}=\E\left[Z_{i}^{2}\right], M_{ij}=\E\left[Z_iZ_j\right], i,j\in \{1,2\}$. The second moment of $Y(\lambda)$ can be expanded in terms of
the entries of $M$ as 
\begin{align*}
\E\left[Z(\lambda)^2\right] = \begin{bmatrix}
\lambda & 1-\lambda
\end{bmatrix}\cdot M \cdot \begin{bmatrix}
\lambda\\ 1-\lambda 
\end{bmatrix}= \lambda^2 (M_{11} + M_{22} - 2M_{12}) +
2\lambda (M_{12} - M_{22}) + M_{22}.
\end{align*}
Being a quadratic function in $\lambda$, the above is minimized at both $\lambda=0, 1$ if and only if it is a constant function. Thus, $M_{11} =M_{12}= M_{22} \equiv m$ holds and we have
$\E\left[ (Z_1 - Z_2)^2\right] = M_{11}+M_{22}-2M_{12}=0$. This concludes the proof of Lemma \ref{lem:avghelps}.

\section{Gaussian Models}\label{sec:gaussian}
We now consider the Gaussian sampling model in Example \ref{eg:gaussian}. For this specific model, the posterior update can be written explicitly, which can also be simulated. For simplicity, we only consider when $d=1$ and the prior $\nu =\normal(0,1)$, but all of our discussion carries through the general case with little modification.
\subsection{Analysis of the posterior distribution}
Recall that we denoted by $H_i(t):=\left(S_i,(Y_j(s))_{j\in \partial i, s<t}\right)$ the collection of messages agent $i$ has received up to time $t$. Then, it is straightforward to inductively deduce that $\{H_i(t)\}_{i\in V}$ is jointly Gaussian conditional on $\theta$. Moreover, the mean and the covariance can be computed iteratively.

\begin{lemma}[Explicit formula for the Gaussian case]\label{lem:gaussian}
For any $t\geq 0$, conditional on $\theta$, $\{H_i(t)\}_{i\in V}$ is jointly Gaussian with mean $\E[H_i(t)\given \theta]=\theta \mu_i(t)$ for some $\mu_i(t)\in \R^{t\cdot|\partial i|+1}$. Denote the covariance matrices by $\Sigma_{ij}(t):=\Cov(H_i(t),H_j(t)\given \theta)\in \R^{(t\cdot|\partial i|+1)\times(t\cdot|\partial j|+1)}$. Then, the posterior of agent $i$ at time $t$ is given by
\begin{equation}\label{eq:gaussian:posterior}
\begin{split}
    &\nu_i(t) = \normal\left(X_i(t),\sigma_i(t)\right),\quad\textnormal{where}\\
    &X_i(t):=\frac{\mu_i(t)^\sT \Sigma_{ii}(t)^{-1}H_i(t)}{\mu_i(t)^\sT \Sigma_{ii}(t)^{-1}\mu_i(t)+1}\quad\textnormal{and}\quad\sigma_i(t):=\frac{1}{\mu_i(t)^\sT \Sigma_{ii}(t)^{-1}\mu_i(t)+1}.
\end{split}
\end{equation}
Moreover, $\mu_i(t+1)$ and $\Sigma_{ij}(t+1)$ can be calculated from $\mu_i(t)$ and $\Sigma_{ij}(t)$ as follows. Note that by definition, $H_i(t+1)=\left(H_i(t), (Y_j(t))_{j\in \partial i}\right)$. Thus, in order to calculate $\mu_i(t+1)$ and $\Sigma_{ij}(t+1)$ from the previous time, it suffices to calculate the mean and the covariance structure for the samples $\{Y_i(t)\}_{i\in V}$ and history $\{H_i(t)\}_{i\in V}$, whose formula is given as follows: for $i,j\in V$, we have
\begin{equation}\label{eq:gaussian:update}
\begin{split}
    \E\left[Y_i(t)\given \theta\right]&= \frac{\mu_i(t)^{\sT}\Sigma_{ii}(t)^{-1}\mu_i(t)}{\mu_i(t)^{\sT}\Sigma_{ii}(t)^{-1}\mu_i(t)+1},\\
    \Cov\left(H_i(t),Y_j(t)\given \theta\right)&= \frac{\Sigma_{ij}(t)\Sigma_{jj}(t)^{-1}\mu_j(t)}{\mu_j(t)^{\sT}\Sigma_{jj}(t)^{-1}\mu_j(t)+1},\\
    \Cov\left(Y_i(t),Y_j(t)\given \theta \right)&=\frac{\mu_i(t)^{\sT}\Sigma_{ii}(t)^{-1}\Sigma_{ij}(t)\Sigma_{jj}(t)^{-1}\mu_j(t)}{\left(\mu_i(t)^{\sT}\Sigma_{ii}(t)^{-1}\mu_i(t)+1\right)\left(\mu_j(t)^{\sT}\Sigma_{jj}(t)^{-1}\mu_j(t)+1\right)}\\
    &\quad\quad+\frac{\one(i=j)}{\mu_i(t)^{\sT}\Sigma_{ii}(t)^{-1}\mu_i(t)+1}.
\end{split}
\end{equation}
\end{lemma}

\begin{proof}
Let $\theta\in \normal(0,1)$ and consider a random vector $X\in \R^d$ whose distribution is given by $X\given \theta \sim \normal(\theta v, \Sigma)$ for some $v\in \R^d$ and $\Sigma \in \R^{d\times d}$. Denote the conditional density of $\theta\given X$ with respect to Lebesgue measure $\mu$ by $p(\theta\given X):=\frac{\rd \law(\theta\given X)}{\rd \mu}$. The densities $p(\theta\given X), p(\theta)$ are similarly defined. Then, by Bayes rule,
\begin{equation*}
    p(\theta\given X)\propto p(X\given \theta)\cdot p(\theta)\propto \exp\left(-\frac{1}{2}(X-\theta v)^{\sT}\Sigma^{-1}(X-\theta v)-\frac{\theta^2}{2}\right).
\end{equation*}
Thus, completing squares, we have that $\theta\given X$ is normally distributed:
\begin{equation}\label{eq:conditional:normal}
    \law(\theta\given X)= \normal\left(\frac{v^{\sT}\Sigma^{-1}X}{v^{\sT}\Sigma^{-1}v+1}, \frac{1}{v^{\sT}\Sigma^{-1}v+1}\right).
\end{equation}
Therefore, since $H_i(t)\sim \normal\left(\theta \mu_i(t),\Sigma_{ii}(t)\right)$ by our assumption, the expression for $\nu_i(t)\equiv \law\left(\theta\given H_i(t)\right)$ in \eqref{eq:gaussian:posterior} is immediate from \eqref{eq:conditional:normal}. Moreover, recall that agent $i$ samples $Y_i(t)\given H_i(t)\sim \nu_i(t)=\normal\left(X_i(t),\sigma_i(t)\right)$ conditionally independent of all other random variables at time $t$. Thus, we can write
\begin{equation}\label{eq:Y}
    Y_i(t)=X_{i}(t)+\sqrt{\sigma_i(t)}Z_i(t)=\frac{\mu_i(t)^\sT \Sigma_{ii}(t)^{-1}H_i(t)}{\mu_i(t)^\sT \Sigma_{ii}(t)^{-1}\mu_i(t)+1}+\frac{Z_i(t)}{\sqrt{\mu_i(t)^\sT \Sigma_{ii}(t)^{-1}\mu_i(t)+1}},
\end{equation}
where $Z_i(t)$ is fresh i.i.d standard normal independent of all other random variables. From \eqref{eq:Y}, it is straightforward to see \eqref{eq:gaussian:update}.
\end{proof}
Therefore, by Theorem \ref{thm:agreement} and \ref{thm:eff:learning}, we have that for agent $i$, almost surely,
\begin{equation}\label{eq:gaussian:posterior:limit}
    \nu_i(t)\wto \law(\theta\given S_1,...,S_n)=\normal\left(X(\infty),\sigma(\infty)\right),
\end{equation}     
\[
X(\infty):=\frac{\sum_{i=1}^{n}a_i S_i}{\sum_{i=1}^{n} {a_i^2}+1}\textnormal{ and }\sigma(\infty):= \frac{1}{\sum_{i=1}^{n} {a_i^2}+1},
\]
where $\nu_i(t)$ takes the explicit form \eqref{eq:gaussian:posterior}. In particular, we have $\lim_{t\to\infty}\sigma_i(t)=\sigma(\infty)$ and $\lim_{t\to\infty} X_i(t)=X(\infty)$, a.s.. Now, observe that $(X_i(t))_{t\geq 0}=\{\E\left[\theta \given \HH_i(t)\right]\}_{t\geq 0}$ is a Doob martingale with variance $\Var(X_i(t))=\sigma_i(t)(1-\sigma_i(t))$ from its definition \eqref{eq:gaussian:posterior}. Thus, $(X_i(t))_{t\geq 0}$ is a $L^2$ bounded martingale and by martingale convergence, we have the following corollary.
\begin{cor}\label{cor:gaussian}
For every agent $i\in V$, we have $\lim_{t\to\infty}\sigma_i(t)=\sigma(\infty)$ and $\lim_{t\to\infty}X_i(t)=X(\infty)$, a.s. and in $L^2$.
\end{cor}

\subsection{Numerical experiments}\label{subsec:numerical}
We numerically examine the convergence of $X_i(t)$ and $\sigma_i(t)$ in Corollary \ref{cor:gaussian} by simulating the sampling model on small graphs. First, we consider the simplest case possible when there are $2$ agents communicating with each other. Figure \ref{fig:edge} below reports a realization of the posterior means $\{X_{i}(t)\}_{t\geq 0}, i=1,2,$ when the signal strengths are $a_1=1$ and $a_2=2$. Indeed, after $t=500$, we see that the posterior means for both of the agents are close to the Bayes posterior mean $X(\infty)$. 

Next, we compare the cases where $G$ is a clique or a cycle with $n=7$ agents. We take the signal structure to be $a_i=i, 1\leq i \leq 7$. For the cycle case, we assume that the agent $i$ is connected to agents $i-1$ and $i+1$ (here, the numbers are calculated mod $7$). Figure \ref{fig:cycle:clique} reports the posterior variances for times up to $t=20$, which are recursively computed by equations \eqref{eq:gaussian:posterior} and \eqref{eq:gaussian:update}. Note that posterior variances monotonically decreases, and this is expected from \eqref{eq:gaussian:posterior} since $\{X_i(t)\}_{t\geq 0}$ being a martingale has increasing variance, so $\mu_i(t)^{\sT}\Sigma_ii(t)^{-1}\mu_i(t)$ is increasing in $t\geq 0$. Moreover, as one would expect, the fully connected network has faster rate of convergence to the variance of the Bayes posterior $\sigma(\infty)$.
\begin{figure}[ht]
\phantom{A}\hspace{-0.5cm}
\includegraphics[width=0.5\textwidth]{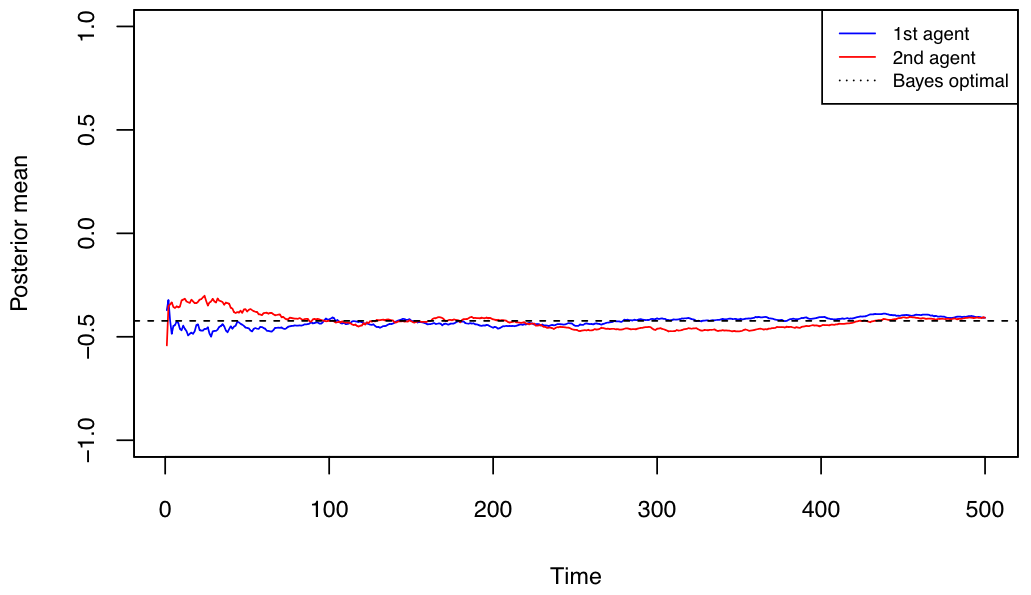}
\includegraphics[width=0.5\textwidth]{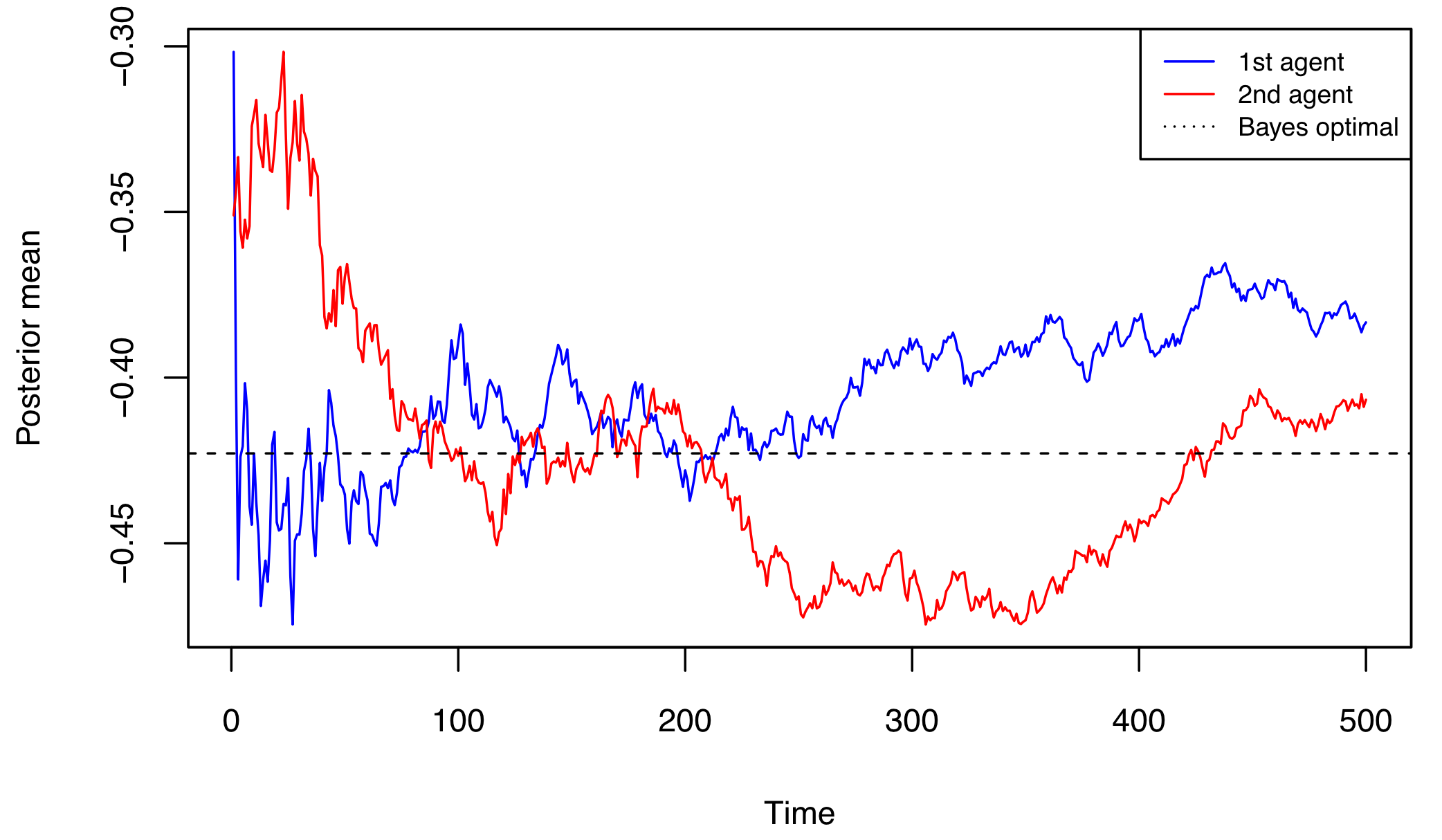}
    \caption{{\footnotesize A realization of posterior means $(X_i(t))_{0\leq t\leq 500}, i=1,2,$ when $G$ is a single edge with Gaussian environment. Left: $(X_i(t))_{0\leq t\leq 500}, i=1,2,$ at the scale of $[-1,1]$, whose scale matches the variance of the signals. Right: the same plot, but zoomed in so that the fluctuations are better seen. Blue line correspond to the first agent with signal strength $a_1=1$ and red line corresponds to the second agent with signal strength $a_2=2$. Here, $\theta$ drawn from $\normal(0,1)$ was $\theta \approx-0.502$ and the $S_1, S_2$ drawn respectively from $\normal(\theta, 1)$ and $\normal(2\theta,1)$ was $S_1\approx-0.371$ and $S_2\approx-1.08$. Note that at time $t=0$, the posterior mean is $X_1(0)=S_1/a_1\approx-0.371$ and $X_2(0)=S_2/a_2\approx -0.542$.  Horizontal dashed line is the mean of the Bayes posterior $\E[\theta \given S_1,S_2]=\frac{a_1 S_1+a_2 S_2}{a_1^2+a_2^2+1}\approx -0.423$. At time $t=500$, $X_1(500)\approx -0.406$ and $X_2(500) \approx -0.407$.}}
    \label{fig:edge}
\end{figure}

\begin{figure}[ht!]
    \centering
    \includegraphics[width=0.75\textwidth]{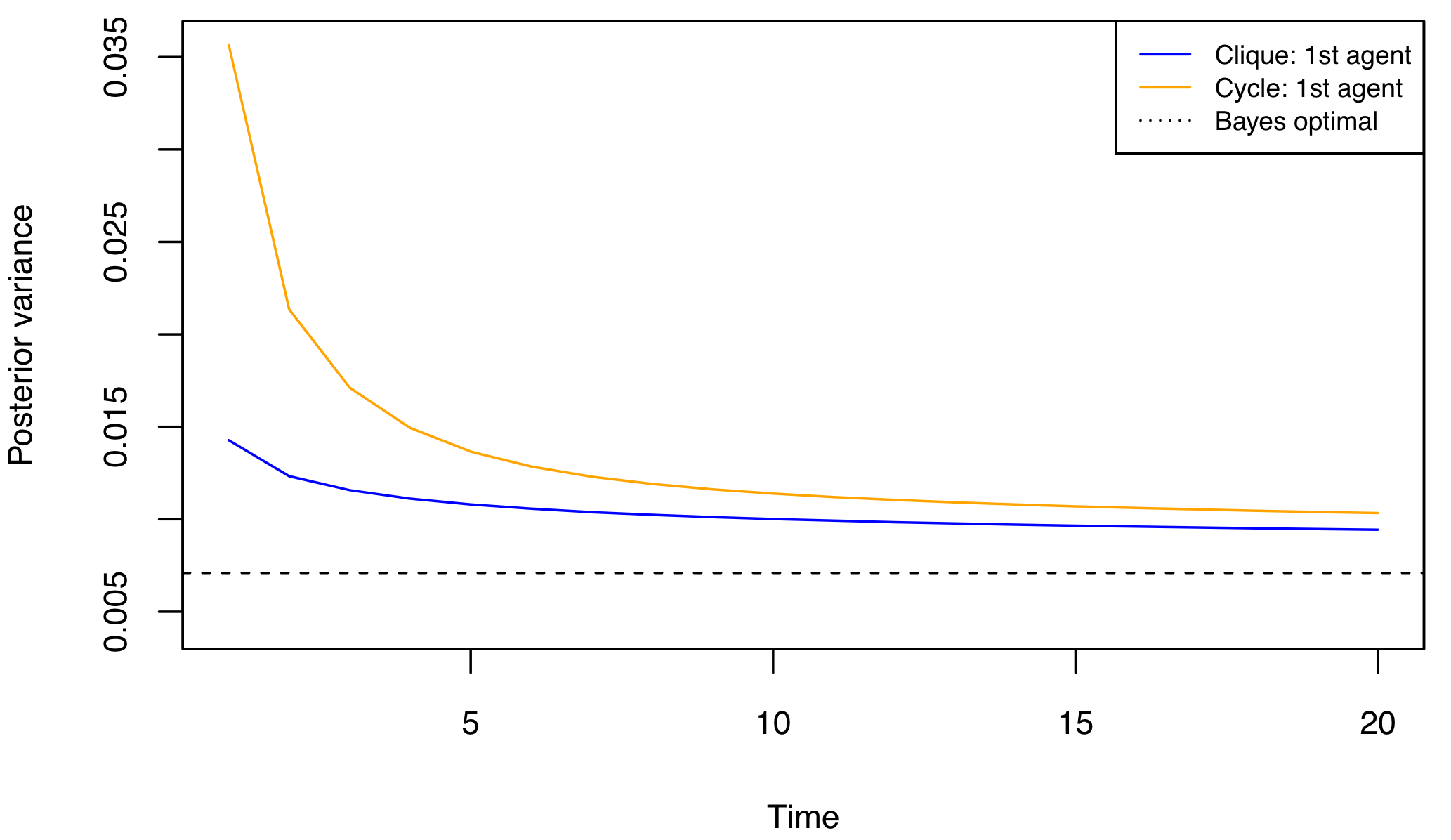}
    \caption{{\footnotesize Posterior variances for the cases where $G$ is a clique or a cycle with $7$ agents with Gaussian environment up to time $t=20$. Here, signal structure is taken to be $a_i= i, 1\leq i \leq 7$. Blue line corresponds to the first agent in the fully connected network and orange line corresponds to the first agent in the cycle network. Horizontal dashed line is the variance of the Bayes posterior $(\sum_{i=1}^{7}i^2+1)^{-1}\approx 0.0071$.}}
    \label{fig:cycle:clique}
\end{figure}

\end{document}